\documentclass[11pt, letterpaper]{amsart}
\usepackage{graphicx, amssymb, color}
\usepackage{amsmath}

\newtheorem{thm}{Theorem}[section]
\newtheorem{cor}[thm]{Corollary}
\newtheorem{lem}[thm]{Lemma}
\newtheorem{prop}[thm]{Proposition}
\theoremstyle{definition}
\newtheorem{defn}[thm]{Definition}

\newtheorem{stasss}[thm]{Standing Assumption}

\theoremstyle{remark}
\newtheorem{rem}[thm]{Remark}

\numberwithin{equation}{section}

\newcommand{\abs}[1]{\left\vert#1\right\vert}
\newcommand{\set}[1]{\left\{#1\right\}}
\newcommand{\Real}{\mathbb R}
\newcommand{\Natural}{\mathbb N}

\newcommand{\such}{\, | \, }
\newcommand{\prob}{\mathbb{P}}

\newcommand{\expec}{\mathbb{E}}

\newcommand{\basis}{(\Omega,  \, (\F_t)_{t \in \Real_+}, \, \prob)}

\newcommand{\F}{\mathcal{F}}

\newcommand{\cadlag}{c\`adl\`ag\,}

\newcommand{\limsupn}{\limsup_{n \to \infty}}

\newcommand{\limn}{\lim_{n \to \infty}}

\newcommand{\plim}{{\prob \textrm{-} \lim}}
\newcommand{\plimn}{\plim_{n \to \infty}}

\newcommand{\pare}[1]{\left(#1\right)}
\newcommand{\bra}[1]{\left[#1\right]}
\newcommand{\cbra}[1]{\left\{#1\right\}}
\newcommand{\dbra}[1]{[\kern-0.15em[ #1 ]\kern-0.15em]}
\newcommand{\dbraco}[1]{[\kern-0.15em[ #1 [\kern-0.15em[}
\newcommand{\dbraoc}[1]{]\kern-0.15em] #1 ]\kern-0.15em]}

\newcommand{\dfn}{\, := \,}

\newcommand{\indic}{\mathbb{I}}

\newcommand{\tsigma}{\widetilde{\sigma}}

\newcommand{\cL}{\mathcal{L}}

\newcommand{\D}{\mathcal{D}}

\newcommand{\fC}{\mathfrak{C}}
\newcommand{\fs}{\mathfrak{s}}
\newcommand{\fv}{\mathfrak{v}}
\newcommand{\ey}{{}^\epsilon \kern-0.15em Y}
\newcommand{\ev}{{}^\epsilon \kern-0.15em V}
\newcommand{\eM}{{}^\epsilon \kern-0.27em M}
\newcommand{\wt}[1]{\widetilde{#1}}

\begin{document}

\title[Valuation equations for stochastic volatility models]{Valuation equations for stochastic volatility models}

\author[]{Erhan Bayraktar}
\thanks{E. Bayraktar is supported in part by the National Science Foundation under an applied mathematics research grant and a Career grant, DMS-0906257 and DMS-0955463, and in part by the Susan M. Smith Professorship.}
\address[Erhan Bayraktar]{Department of Mathematics, University of Michigan, 530 Church Street, Ann Arbor, MI 48104, USA}
\email{erhan@umich.edu}
\author[]{Constantinos Kardaras}\thanks{C. Kardaras is supported in part by the National Science Foundation under grant number DMS-0908461.}
\address[Constantinos Kardaras]{Department of Mathematics and Statistics, Boston University, 111 Cummington Street, Boston, MA 02215, USA}
\email{kardaras@bu.edu}
\author[]{Hao Xing }
\address[Hao Xing]{Department of Statistics, London School of Economics and Political Science, 10 Houghton st, London, WC2A 2AE, UK}
\email{h.xing@lse.ac.uk}
\thanks{We wish to thank A. Kuznetsov, H. Pham, N. Touzi, M. Urusov, G. \v{Z}{i}tkovi\'{c}, and most notably T. Salisbury and M. S\^{i}rbu for helpful discussions. The third author wishes to thank the Fields Institute, where part of this work was carried out. We are grateful to two anonymous referees and the Associate Editor for their valuable comments, which helped us improve this paper.}


\begin{abstract}

We analyze the valuation partial differential equation for European contingent claims in a general framework of stochastic volatility models where the diffusion coefficients may grow faster than linearly and degenerate on the boundaries of the state space. We allow for various types of model behavior: the volatility process in our model can potentially reach zero and either stay there or instantaneously reflect, and the asset-price process may be a strict local martingale.
Our main result is a necessary and sufficient condition on the uniqueness of classical solutions to the valuation equation: the value function is the unique nonnegative classical solution to the valuation equation among functions with at most linear growth if and only if the asset-price is a martingale.
\end{abstract}

\date{December 12, 2011}

\maketitle

\section{Introduction}\label{sec: intro}

Unlike the Black-Scholes model, stochastic volatility models are incomplete. For the purpose of valuing contingent claims written on the underlying asset, one typically postulates a diffusion model for the asset price and its volatility, formulated under a risk-neutral measure that is calibrated to market data. Due to the Markovian structure of stochastic volatility models, valuing a European contingent claim boils down to determining a value function, which is plainly the expectation (under the chosen risk-neutral measure) of the terminal payoff evaluated at the market's current configuration, including the current asset price, the level of the factor that drives the volatility, as well as the time-to-maturity. A way to determine this value function is by solving a partial differential equation (PDE), which we call the \emph{valuation equation}, heuristically derived by formally applying It\^{o}'s formula and utilizing a martingale argument.

However, as was pointed out in \cite{Heath-Schweizer}, it is surprisingly tricky to rigorously prove the aforementioned heuristic argument.
To begin with,
valuation equations in stochastic volatility models are typically degenerate on the boundaries of state space.
Therefore, the assumptions in standard versions of the Feynman-Kac formula (see e.g.  \cite[Chapter 6]{friedman-stochastic}) are not satisfied for many stochastic volatility models used in practice.

Moreover, the asset-price process in stochastic volatility models can be a strict local martingale; see \cite{Sin}, \cite{Andersen-Piterbarg}, \cite{Lions-Musiela}, \cite{Hobson}, and \cite{Lewis}.
(The loss of the martingale property relates to the notion of stock price \emph{bubbles}; see \cite{Heston-Loewenstein-Willard}, \cite{Cox-Hobson},
\cite{JPS} and \cite{Jarrow-Protter-Shimbo}. Similar situations have also been studied in markets without local martingale measures;
see \cite{DFern-Kar09}, \cite{Johannes}, and \cite{Fernholz-Karatzas-survey}.)
An important consequence of losing the martingale property, mentioned in \cite{Heston-Loewenstein-Willard}, is that the valuation equation may have multiple solutions. 
The strict local martingale property of the asset price may induce faster-than-quadratic growth in coefficients for valuation equations,
while the standard theory of either classical or viscosity solutions usually assume at most quadratic growth in coefficients before second derivative terms, see e.g. \cite{friedman-stochastic} and \cite{Fleming-Soner}.

In this paper, we study a general framework of stochastic volatility models, where coefficients are H\"{o}lder continuous, degenerate on boundaries of state space, and asset-price volatility coefficient may grow faster than linearly. In these models, we focus on the following questions:
\begin{enumerate}
	\item[(Q1)] How should one formulate the concept of a solution of the valuation equation (regarding smoothness and boundary conditions) in order to ensure that the value function is one such solution?
    \item[(Q2)] Given that (Q1) has been answered, what is the necessary and sufficient condition for the value function to be the unique solution in a certain class of candidate functions?
  \end{enumerate}

Equations with degenerating coefficients have been studied extensively, see e.g. \cite{Kohn-Nirenberg} and \cite{MR0457908}. More recently,
in order to study the free boundary of the porous medium equation, \cite{Daskalopoulous-Hamilton} and \cite{Daskalopoulous-Lee} investigated a linear degenerate equation, which is exactly the valuation equation in the Heston model.
Existence and uniqueness have been proven in a weighted H\"{o}lder space and regularity of solutions close to the degenerate region has also been
established in this case. In mathematical finance literature, existence and uniqueness questions for degenerate equations have been tackled for the case of local volatility models in \cite{Janson-Tysk}, \cite{Ekstrom-Tysk-bubble}, and \cite{Bayraktar-Xing}; for the case of interest rate models in \cite{Ekstrom-Tysk-term}. For stochastic volatility models, these questions have been discussed in \cite{Heath-Schweizer}, \cite{Ekstrom-Tysk-stochvol}, and \cite{Feehan}.
However valuation equations in general stochastic volatility models, whose coefficients may grow faster-than-linearly, have not been well understood yet.

Another natural analytical tool for analyzing degenerate equations is the theory of viscosity solutions. In this framework, it is usually assumed that model coefficients are globally Lipschitz in the state space (see e.g. \cite{Fleming-Soner} and \cite{Barles-et-al}). Therefore, standard techniques need to be extended to study equations whose coefficients are locally Lipschitz in the interior of the state space. See \cite{Amadori} and \cite{CDP2010} for recent developments in this direction. In these two papers, it is assumed that boundaries of the state space are not reached by the state process starting from the interior. To allow for various types of model behaviors, we study the situation where the state process can potentially reach zero and either stay there or instantaneously reflect. Moreover, comparing to the sufficient conditions for uniqueness of solutions to valuation equations in \cite{Amadori} and \cite{CDP2010},  our goal is to identify a necessary and sufficient condition for uniqueness or, equivalently, for the failure of uniqueness.

Rather than employing the analytical methods described above, some authors chose to use probabilistic methods to analyze degenerate equations. In Feller's seminal work \cite{Feller}, semi-group techniques were employed to study one-dimensional PDEs. According to the type of boundary points, different boundary conditions were specified to ensure the uniqueness of solutions. See \cite{Campiti} and references therein for recent development in this direction. On the other hand, \cite{Stroock-Varadhan} and \cite{MR606800} used martingale techniques to analyze these types of problems.

In this paper we employ a combination of probabilistic and analytical techniques to give a necessary and sufficient condition for the uniqueness of solutions to the valuation equation. To the best of our knowledge, this condition had not been identified in the literature. To derive this condition, our strategy is the following: First we identify a necessary and sufficient condition for uniqueness in the class of stochastic solutions (see Section~\ref{sec: stoch soln}), a notion introduced by Stroock and Varadhan in  \cite{Stroock-Varadhan}. Then, in the analytical part of the paper, we show that the value function is a classical solution (in the sense of Definition \ref{def: classical soln}), and that classical solutions are stochastic solutions, see Section~\ref{sec: proof of main thm}.

  Our main contributions can be stated as follows:
\begin{itemize}
\item The stochastic volatility models we analyze have degenerate coefficients on boundaries of the state space. Moreover, the volatility coefficient of the asset price is allowed to have faster than linear growth.
\item The volatility process can potentially reach zero. This extends results in \cite{Heath-Schweizer}, \cite{Amadori}, and \cite{CDP2010}. We classify the local behavior of the volatility process near zero and introduce notions of classical solutions in each scenario to answer (Q1).
\item The asset-price process can be a strict local martingale. We give an analytic condition which is necessary and sufficient for the martingale property of the asset price. This condition generalizes results in \cite{Lions-Musiela} and it is a stronger version of the condition in \cite{Sin}. Meanwhile, it is exactly the loss of martingale property that leads us to an answer to (Q2): uniqueness holds in the class of at most linear growth functions \emph{if and only if} the asset-price process is a martingale. This result complements the uniqueness result in \cite{Ekstrom-Tysk-stochvol}.
\end{itemize}
Our main result is presented in Theorems~\ref{thm: existence} and \ref{thm: uniqueness}. The former shows that the value function is the smallest nonnegative classical solution of the valuation equation, whereas the latter characterizes exactly when the valuation equation has a unique solution in a certain class of functions. Together with the results in Section~\ref{sec: martingale}, this gives us an analytic characterization of the uniqueness of solutions to the valuation equation.

\smallskip

The remainder of the paper is organized as follows. Our main results are presented in Section \ref{sec: main results}. The analytic necessary and sufficient condition on the martingale property of the asset-price process is explored in Section \ref{sec: martingale}. This provides an analytic characterization of the uniqueness obtained in Theorem \ref{thm: uniqueness}. Our main findings are proved progressively in Sections \ref{sec: value function}, \ref{sec: stoch soln} and \ref{sec: proof of main thm}. In particular, the notion of a stochastic solution is introduced in Section~\ref{sec: stoch soln} to bridge the analytic and the probabilistic properties of solutions to the valuation equation.

\section{Main results}\label{sec: main results}

\subsection{The model}
All stochastic processes in the sequel are defined on a filtered probability space $\basis$, satisfying the usual conditions. All relationships between random variables are understood in the $\prob$-a.s. sense. We denote $\Real_+ = [0,\infty)$ and $\Real_{++} = (0,\infty)$.

The following stochastic volatility model will be considered, written for the time being formally in differential form:
\begin{align*}
 dS_t &= S_t \, b(Y_t) \,dW_t, & S_0 = x \in \Real_{+}, \label{eq:sde-S} \tag{STOCK}\\
 dY_t &= \mu(Y_t) \,dt + \sigma(Y_t) \,dB_t,  & Y_0 = y \in \Real_{+}. \label{eq:sde-Y} \tag{VOL}
\end{align*}
Above, $W$ and $B$ are two standard Wiener processes with constant instantaneous correlation $\rho \in (-1, 1)$. In this model, the asset price is modeled by the dynamics of $S$, whose volatility is driven by an auxiliary process $Y$. To simplify notation, we assume the instantaneous short rate to be zero; we note, however, that all our results carry for the case of nonzero constant short rate, with obvious modifications. The dynamics in \eqref{eq:sde-S} imply that $\prob$ is a local martingale measure for the  asset-price process $(S_t)_{t \in \Real_+}$. As mentioned in the Introduction, we allow for the possibility that the latter process is a strict local martingale.

\begin{stasss}\label{ass: coeffs}
It will be tacitly assumed throughout the paper that the coefficients of \eqref{eq:sde-S} and \eqref{eq:sde-Y} satisfy the following:
\begin{enumerate}
 \item[(i)] The function $\mu: \Real_+ \rightarrow \Real$ satisfies $\mu(0) \geq 0$. The functions $\sigma, b: \Real_+ \rightarrow \Real_+$ are strictly positive on $\Real_{++}$, and satisfy $\sigma(0)=b(0)=0$. Also, $\mu$ and $\sigma$ have at most linear growth, i.e., there exists a positive constant $C$ such that
 \begin{equation}
  |\mu(y)| + \sigma(y) \leq C(1+y) \quad \text{ for } y\in \Real_+.\label{eq: mu sigma linear}
 \end{equation}
 \item[(ii)] $\mu$, $\sigma^2$, $b^2$, and $b \sigma$ are continuously differentiable on $\Real_+$ with locally $\alpha$-H\"{o}lder continuous derivatives for some $\alpha\in (0,1]$. Moreover, $(b^2)'$ has at most polynomial growth, i.e., there exist positive constants $C$ and $m$ such that
  \begin{equation}\label{eq: b poly}
    |(b^2)'(y)| \leq C\pare{1+y^m} \quad \text{ for } y\in \Real_+.
  \end{equation}
\end{enumerate}
\end{stasss}

Assumption \ref{ass: coeffs} (i) implies that \eqref{eq:sde-Y} admits a unique nonexplosive and nonnegative strong solution $Y^y$. Under Assumption \ref{ass: coeffs} (ii), $\sigma$ and $b$ may not be Lipschitz continuous on $\Real_+$, but both of them are locally $1/2$-H\"{o}lder continuous. Moreover, $b$ could grow faster than linearly.

\begin{rem}\label{remark: models}
The standing assumptions above are satisfied by most diffusion stochastic volatility models that are used in practice. For example:
\begin{itemize}
	\item in the \emph{Hull-White} model \cite{Hull-White}, $\mu (y) = ay$ with $a<0$, $\sigma(y) = \sigma y$ with $\sigma>0$;
	\item in the \emph{Heston} model \cite{Heston}, $\mu (y) = \mu_0-ay$ with $\mu_0>0$ and $a>0$, $\sigma(y)=\sigma\sqrt{y}$  with $\sigma>0$;
	\item in the \emph{GARCH(1,1)} model, $\mu(y) = \mu_0-ay$ with $\mu_0>0$ and $a>0$, $\sigma(y)=\sigma y$ with $\sigma>0$.
\end{itemize}
In all of the above models, $b(y) = \sqrt{y}$ for $y \in \Real_+$. When $b(y)=y$ for $y \in \Real_+$, we have the model proposed in \cite{Wiggins}.
\end{rem}

For given $(x, y) \in \Real_+^2$, the solution of \eqref{eq:sde-S} is given by the process $S^{x, y} \dfn x H^y$, where
\begin{equation}\label{eq: def H}
 H^y \dfn \exp\left\{\int_0^\cdot b(Y^y_t)\,dW_t - \frac12 \int_0^\cdot b^2(Y^y_t) \, dt\right\}.
\end{equation}
As $b$ is locally bounded on $\Real_+$ and $Y^y$ is nonexplosive, $\int_0^t b^2(Y_u)\, du<\infty$, hence $H^y_t > 0$, for any $t \in \Real_+$.
Define $\tau^y_0 := \inf\set{t \in \Real_{++} \such Y^y_t =0}$. It is possible that $\prob[\tau^y_0 < \infty] > 0$. In this case:
\begin{itemize}
	\item when $\mu(0)=0$, $Y^y_t = 0$ for $\tau^y_0 \leq t < \infty$, thus the point $0$ is \emph{absorbing};
	\item when $\mu(0)>0$, $Y^y$ is lead back into $\Real_{++}$ after $\tau^y_0$, and the point $0$ is \emph{instantaneously reflecting} (see Definition~3.11 in  \cite[Chapter VII]{Revuz-Yor})
\end{itemize}
Lemma \ref{lemma: local time Y} below shows that the local time of $Y^y$ at point $0$ is actually zero in the latter case.

\subsection{The valuation equation}
We consider a European option with a payoff function $g$ which satisfies the following assumption:

\begin{stasss}\label{ass: payoff}
 The function $g: \Real_+ \rightarrow \Real_+$ is nonnegative, continuous, and has at most linear growth, i.e., there exists a positive constant $M$ such that $g(x) \leq M(1+x)$ for $x\in \Real_+$. 
\end{stasss}

Recall that $g$ is of \emph{linear growth}, if $\eta:= \limsup_{x\rightarrow \infty} g(x)/x >0$, otherwise $g$ is of \emph{strictly sublinear growth}. Let us consider the smallest concave, nonnegative, and nondecreasing function $h$ that dominates $g$. It has been shown in \cite{CPH} that $h$ is the super-replication price for the payoff $g$. It is clear that $h(x)\leq M(1+x)$ for $x\in \Real_+$. Moreover, Lemma \ref{lem: h growth} below shows that $h$ has linear or strictly sublinear growth whenever $g$ does.

The value function $u : \Real_+^3 \rightarrow \Real_+$ of a European option with the payoff $g$ is defined via
\begin{equation*}\label{eq: def u}
 u(x, y, T) := \expec \bra{g\pare{S^{x, y}_T}}, \quad \text{ for } (x, y, T) \in \Real_+^3.
\end{equation*}
It is dominated by $h$. Indeed,
\begin{equation}\label{eq: u growth}
 u(x,y,t) = \expec\bra{g(S^{x,y}_t)} \leq \expec\bra{h(S_t^{x,y})} \leq h\pare{\expec[S^{x,y}_t]} \leq h(x), \quad (x,y,t)\in \Real_+^3.
\end{equation}
For $(x, y, T) \in \Real_+^3$, define a process $U^{x, y, T} = (U^{x, y, T}_t)_{t \in [0, T]}$ via $U^{x, y, T}_t := u(S^{x,y}_t, Y^{y}_t, T-t)$ for $t \in [0, T]$. The Markov property of $(S^{x, y}, Y^y)$ gives
\begin{equation}\label{eq: mart u}
U^{x, y, T}_t = \expec \bra{g\pare{S^{x, y}_T} \, |\, \F_t}, \quad  t \in[0,T].
\end{equation}
As $\expec \bra{g\pare{S^{x, y}_T}} < \infty$,  $U^{x, y, T}$ is clearly a martingale on $[0, T]$.

If $u$ is sufficiently smooth (at the moment, we are being intentionally vague on this point; we shall have more to say in Theorem \ref{thm: existence}), a formal application of It\^o's formula implies that the value function $u$ is expected to solve the valuation equation
\begin{equation}\label{eq: pricing eq} \tag{BS-PDE}
\begin{split}
 & \partial_T v (x,y,T) = \cL v (x,y,T), \quad (x,y,T) \in \Real_{++}^3,\\
 & v(x,y,0) = g(x), \hspace{2cm} (x,y)\in \Real_+^2,
\end{split}
\end{equation}
in which
\[
\cL := \mu(y) \partial_y + \frac12 b^2(y) x^2 \partial^2_{xx} +  \frac12 \sigma^2(y) \partial^2_{yy}  + \rho b(y) \sigma(y) x \partial^2_{x y}
\]
is the infinitesimal generator of $(S, Y)$. Since $b$ can grow faster than linearly, coefficients before second order derivatives above can grow faster than quadratically.

Further conditions are usually supplied to \eqref{eq: pricing eq} to guarantee that $u$ is the unique solution in a certain class of functions. 
To motivate these conditions, consider a solution $v$ to \eqref{eq: pricing eq}. If it is to be identified with $u$, it is clearly necessary that the process $V^{x, y, T} = (V^{x, y, T}_t)_{t \in [0, T]}$, defined via $V^{x, y, T}_t := v(S^{x,y}_t, Y^{y}_t, T-t)$ for $t \in [0, T]$ and $(x,y,T)\in \Real_{++}^3$, is at least a local martingale on $[0,T]$. Given $v\in C^{2,2,1}(\Real^3_{++})$, It\^{o}'s lemma implies that $V^{x, y, T}$ is a local martingale up to $\tau^y_0 \wedge T$. When $\prob[ \tau^y_0 < T] > 0$, it is reasonable to expect that some boundary condition at $y=0$ is needed to ensure that $V^{x, y, T}$ is still a local martingale after $\tau^y_0$ and up to $T$. When $\mu(0)=0$, the point $0$ is absorbing for $Y^y$. Since $b(0)=0$, we have $(S^{x, y}_t, Y^y_t) = (S^{x, y}_{\tau^y_0}, 0)$ for $\tau_0^y \leq t<\infty$. Therefore, we enforce the following Dirichlet boundary condition,
\begin{equation}\label{eq: boundary cond 2}
 v(x,0,T) = g(x), \quad (x,T)\in \Real^2_{++}.
\end{equation}
When $\mu(0)>0$, the boundary condition restricts the classical solution to the point-wise closure of the following class $\fC$.
\begin{defn}\label{def: class C}
A function $v: \Real_+^3 \rightarrow \Real_+$ is an element of $\fC$ if
\begin{enumerate}
\item[(i)] $v\in C(\Real_+^3) \cap C^{2,2,1}(\Real_{++}^3) \cap C^{0,1,1}(\Real_{++}\times \Real_+ \times \Real_{++})$,
\item[(ii)] $\limsup_{y\downarrow 0} b^2(y) \left|\partial^2_{xx} v(x,y,T)\right|<\infty$ for $(x,T)\in \Real_{++}^2$,
\item[(iii)] $0\leq v(x,y,T) \leq h(x)$ for $(x,y,T)\in \Real_+^3$ and
\item[(iv)] $\partial_T v(x,y,T) = \cL v(x,y,T)$ for $(x,y,T)\in \Real_{++}^3$.
\end{enumerate}
\end{defn}

We say a sequence $(v_n)_{n \geq 0}$ converges to $v$ point-wise, if $\lim_{n\rightarrow \infty} v_n(x,y,T) = v(x,y,T)$ for any $(x,y,T) \in \Real_+^3$. We denote by $\overline{\fC}$ the smallest set containing $\fC$ and closed under the point-wise convergence. Note that element of $\fC$ may not satisfy the initial condition in \eqref{eq: pricing eq}. In Theorem \ref{thm: existence} below, when $Y$ instantaneously reflects at zero, we will use a sequence of functions in $\fC$ with bounded initial conditions to approximate the value function.

Now let us now define what we mean by classical solutions to \eqref{eq: pricing eq}. The definition depends on whether $Y^y$ hits zero in finite time, which is characterized by Feller's test (see e.g. Theorem 5.5.29 in \cite{Karatzas-Shreve-BM}). Since the value function $u$ is nonnegative and dominated by $h$, in order to identify $u$ as a solution to \eqref{eq: pricing eq}, it suffices to consider nonnegative solutions which are dominated by $h$.

\begin{defn}\label{def: classical soln}
A function $v: \Real^3_+ \rightarrow \Real_+$ is called a classical solution (with growth domination $h$), if it satisfies conditions specified in each of the following cases (below, $y$ is arbitrary in $\Real_{++}$):
\begin{enumerate}
 \item[(A)] When $\prob[\tau^y_0=\infty]=1$: $v \in C(\Real_+^3) \cap C^{2,2,1}(\Real_{++}^3)$, $0\leq v\leq h$, and $v$ solves \eqref{eq: pricing eq}.

 \item[(B)] When $\prob[\tau^y_0<\infty] >0$ and $\mu(0)=0$: $v$ satisfies all conditions in Case (A) and  the boundary condition \eqref{eq: boundary cond 2}.

 \item[(C)] When $\prob[\tau^y_0<\infty]>0$ and $\mu(0)>0$: $v\in \overline{\fC} \cap C(\Real^3_+)$ and satisfies the initial condition $v(x,y,0)= g(x)$ on $\Real^2_+$.
\end{enumerate}
\end{defn}

A function $v$ is a super (sub)-solution to \eqref{eq: pricing eq}, if it satisfies properties in the previous definition where both equations in \eqref{eq: pricing eq} and in item (iv) in Definition~\ref{def: class C} are replaced by $\partial_T v \geq \cL v$ ($\partial_T v\leq \cL v$), respectively.

\begin{rem}
In Case (C) of the above definition, any $v\in \overline{\fC}$ satisfies $0\leq v\leq h$ on $\Real_+^3$. Moreover, it is, in fact, an element of $C^{2,2,1}(\Real_{++}^3)$ and solves $\partial_T v =\cL v$ on $\Real_{++}^3$. This is why we call $v$ a classical solution to \eqref{eq: pricing eq} in this case. Indeed, since $v\in \overline{\fC}$, there exists a sequence $\set{v_n}_{n \geq 0}$, with each $v_n \in \fC$, such that they converge to $v$ point-wise. Fix any compact domain $\D \subset \Real_{++}^3$. Since $\set{v_n}_{n \geq 0}$ is uniformly bounded from above by $h$ and the differential operator $\cL$ is uniformly elliptic on $\D$, it then follows from the \emph{interior Schauder estimate} (see e.g. Theorem~15 in \cite{friedman-parabolic} pp. 80) that $v \in C^{2,2,1}(\D')$ for any compact subdomain $\D' \subset \D$ and $v$ solves $\partial_T v = \cL v$ on $\D'$. Then the claim follows since the choice of $\D$ is arbitrary in $\Real_{++}^3$.
\end{rem}

\begin{rem}\label{rem: boundary cond}
 Boundary conditions are specified in Definition~\ref{def: classical soln} to identify the value function $u$ as the unique solution with growth domination $h$ (see Theorem \ref{thm: uniqueness} below). Therefore, even if the value function has certain regularity at boundaries, if these properties are not necessary for the proof of uniqueness, it is not included in Definition \ref{def: classical soln}. This is different from the point of view in \cite{Ekstrom-Tysk-stochvol}, where the value function is shown to satisfy a first order equation (see \eqref{eq: boundary cond 1} below), under additional assumptions on payoffs, no matter whether the process $Y$ visits the boundary or not.
\end{rem}

\subsection{Existence and uniqueness results}
The following are the main results of this paper. Their proofs are given in Section~\ref{sec: proof of main thm}.

\begin{thm}[Existence]\label{thm: existence}
The value function $u$ is a classical solution to \eqref{eq: pricing eq}. Moreover, it is the smallest classical solution.
\end{thm}

\begin{thm}[Uniqueness]\label{thm: uniqueness}
The following two statements hold:
\begin{enumerate}
\item[(i)] When $g$ is of strictly sublinear growth, $u$ is the unique classical solution with growth domination $h$.

\item[(ii)] When $g$ is of linear growth, $u$ is the unique classical solution with growth domination $h$ if and only if the asset price process $S$ is a martingale.
\end{enumerate}
Uniqueness holds if and only if the following comparison result holds. Let $v$ and $w$ be  classical super/sub-solutions with growth domination $h$. If $v(x,y,0) \geq g(x) \geq w(x,y,0)$ for $(x,y) \in \Real_+^2$, then $v \geq w$ on $\Real_+^3$.
\end{thm}

\begin{rem}
 Lemma~\ref{lem: h growth} below shows that $h$ has linear or strictly sublinear growth whenever $g$ does. Then the uniqueness are considered in the class of functions which have the same growth with $g$.
\end{rem}

\begin{rem}
 Our main contribution is the uniqueness theorem. In the classical theory of parabolic PDEs, a sufficient condition to ensure the uniqueness of classical solutions among the class of functions with at most polynomial growth is that coefficients before the second and first order spatial derivatives have at most quadratic and linear growth, respectively; see e.g. Corollary 6.4.4 in \cite{friedman-stochastic}. In stochastic volatility models considered in this paper, Theorem~\ref{thm: uniqueness} shows that uniqueness may fail among functions with at most linear growth if aforementioned growth conditions on coefficients are not satisfied. Multiple solutions are constructed via strict local martingales. Therefore, the martingale property of the asset price, which is characterized analytically in the next section, provides a necessary
 and sufficient condition for the uniqueness of classical solutions. This main result extends results in \cite{Bayraktar-Xing} for local volatility models. As we shall see in Section~\ref{sec: proof of main thm}, the proof of Theorem~\ref{thm: uniqueness} relies on probabilistic arguments. This is in contrast with the analytic approach used in \cite{Ekstrom-Tysk-stochvol}.
\end{rem}

\section{Characterizing the Martingale Property of the  Asset-Price Process}\label{sec: martingale}

In this section, we shall present a necessary and sufficient analytic condition for the martingale property of the asset price process, which is essentially $H^y$ (up to normalization with respect to the initial asset price). Combined with Theorem \ref{thm: uniqueness} (ii), this provides a necessary and sufficient analytic condition for the uniqueness of classical solutions for \eqref{eq: pricing eq} among functions with growth domination $h$.

Let us consider an auxiliary diffusion $\wt{Y}$ governed by the following formal dynamics:
\begin{equation}\label{eq: sde-tY}
  d\wt{Y}_t = \wt{\mu} (\wt{Y}_t) \, dt + \sigma(\wt{Y}_t) \, d B_t, \quad \wt{Y}_0=y,
\end{equation}
where $\wt{\mu} \dfn \mu + \rho b \sigma$. By our standing assumption, $\wt{\mu}$ is locally Lipschitz and $\sigma$ is locally $(1/2)$-H\"{o}lder continuous. Therefore \eqref{eq: sde-tY} has a unique nonnegative strong solution $\wt{Y}^y$, for all $y \in \Real_+$. However, due to the fact that $\wt{\mu}$ is only locally Lipschitz, the solution $\wt{Y}^y$ is defined up to an explosion time $\zeta^y$, and it might be the case that $\prob \bra{\zeta^y < \infty} > 0$. This has important consequences on the stochastic behavior of the asset-price process, as the following result demonstrates.

\begin{prop}\label{prop: mart S}
\begin{equation}\label{eq: expec S}
  \expec \bra{S^{x, y}_T} = x \, \expec \bra{H^y_T} =  x \, \prob \bra{\zeta^y > T}, \quad \text{ for all } (x, y, T) \in \Real_+^3.
 \end{equation}
Moreover, $\prob \bra{\zeta^{y_1} \leq \zeta^{y_2} } = 1$, holds whenever $y_1 \in \Real_+$ and $y_2 \in [0, y_1]$.
\end{prop}

\begin{rem}
 The assumption that $Y$ is nonexploding is essential. Without it, the representation \eqref{eq: expec S} may not hold. See \cite{Mijatovic-Urusov-counterexample} and \cite{Mijatovic-Urusov}.
\end{rem}

\begin{proof}
 Since $Y$ is nonexploding, \eqref{eq: expec S} follows from an argument similar to the one used in the proof of Lemma~4.2 in \cite{Sin}. Also, see Lemma~2.3 in \cite{Andersen-Piterbarg}. The fact that $\prob \bra{\zeta^{y_1} \leq \zeta^{y_2} } = 1$ holds follows from standard comparison theorems for SDEs  --- see e.g. Proposition 5.2.18 of \cite{Karatzas-Shreve-BM}.
 \end{proof}


Whether an explosion of $\wt{Y}$ happens or not is fully characterized by Feller's test, which we now revisit. With a fixed $c \in \Real_{++}$, the scale function $\fs$ for the diffusion described in \eqref{eq: sde-tY} is defined as
\[
 \fs(y)  := \int_c^y \exp\left\{-2 \int_c^\xi \frac{\wt{\mu}(z)}{\sigma^2(z)}  \,dz\right\} \, d \xi, \quad \text{ for } y \in \Real_{++}.
\]
We set
\[
 \fv(y) := 2\int_c^{y} \frac{\fs(y) - \fs(\xi)}{\fs'(\xi) \sigma^2(\xi)} \, d \xi \quad \text{ for } y \in \Real_{++}.
\]
Note that $\fv$ is increasing on $(c,\infty)$. Therefore, $\fv(\infty):=\lim_{y \uparrow \infty} \fv(y)$ is well defined. Feller's test (see e.g. Theorem 5.5.29 in \cite{Karatzas-Shreve-BM}) states that $\prob \bra{\zeta^y < \infty} > 0$ for $y \in \Real_{++}$ if and only if
\begin{equation}\label{feller test}
 \fv(\infty) < \infty.
\end{equation}
As was pointed out in \cite[Section 4.1]{Cherny-Engelbert}, it is sometimes easier to check the following equivalent condition:
\begin{equation}\label{feller test equiv}
 \fs(\infty) < \infty \quad \text{ and } \quad \frac{\fs(\infty) - \fs}{\fs' \sigma^2} \in L^1_{loc}(\infty-),
\end{equation}
where $L^1_{loc}(\infty-)$ denotes the class of functions $f: \Real_+ \rightarrow \Real$ that are Lebesgue integrable on $(y,\infty)$ for some $y>0$.

Combining \eqref{eq: expec S} and the above discussion, one obtains the following corollary of Proposition \ref{prop: mart S}, which is due to \cite{Sin}: \emph{$H^{y}$ is a martingale for all $y \in \Real_{++}$ if and only if \eqref{feller test}  fails to hold (or, equivalently,  if and only if  \eqref{feller test equiv} fails to hold)}. The previous statement implies that $H^{y}$ is a strict local martingale for some, and then all, $y \in \Real_{++}$ if and only if \eqref{feller test} (or \eqref{feller test equiv}) is satisfied. However, given that $H^{y}$ is a strict local martingale, it is not clear whether $H^{y}_{\cdot \wedge T}$ is still a strict local martingale for any $T>0$. The next result has is a stronger statement than the one previously made. Its proof requires some later results of this paper; therefore, we defer it to Section~\ref{sec: value function}.

\begin{prop}\label{prop: S strict local mart}
The following statements are equivalent:
 \begin{enumerate}
 \item $H^y_{\cdot \wedge T}$ is a strict local martingale for some, and then all $(y, T) \in \Real_{++}^2$.
 \item \eqref{feller test} (or, equivalently, \eqref{feller test equiv}) is satisfied.
 \end{enumerate}
\end{prop}

Note that when $H^y$ is a martingale for all $y\in \Real_{++}$, $H^0$ is a martingale as well because of the monotonicity of $\Real_+ \ni y \mapsto \prob[\zeta^y > T]$ in $y$ for fixed $T \in \Real_+$ --- see Proposition~\ref{prop: mart S}.
In view of Proposition \ref{prop: S strict local mart}, when we are referring to the martingale property of the  asset-price process, we mean that $H^y$ is a martingale for all $y \in \Real_+$.

\begin{rem}
Proposition~\ref{prop: S strict local mart} implies that if $H^y$ is going to lose its martingale property eventually, it must lose its martingale property immediately. This result generalizes Theorem~2.4 in \cite{Lions-Musiela}, where a sufficient condition and a different necessary condition are given such that $H^y_{\cdot \wedge T}$ is a strict local martingale for any fixed $T\in \Real_{++}$. Proposition~\ref{prop: S strict local mart} closes the gap between these two conditions in \cite{Lions-Musiela}. When the boundary point $0$ is absorbing, Proposition~\ref{prop: S strict local mart} is contained in the main result of \cite{Mijatovic-Urusov}. However Proposition~\ref{prop: S strict local mart} also treats the case when the boundary point is instantaneously reflecting.

One should note, however, that when the dynamics in the stochastic volatility model are not time-homogeneous, the  asset price may lose its martingale property only at a later time, as can be seen from an example in Section 2.2.1 in \cite{Cox-Hobson}.
\end{rem}

\section{Smoothness of the Value Function}\label{sec: value function}
In this section we shall prove $u\in C(\Real_+^3) \cap C^{2,2,1}(\Real^3_{++})$, as well as Proposition~\ref{prop: S strict local mart}, an important corollary of this result. Let us start with a technical result on the stability of solutions of \eqref{eq:sde-S} and \eqref{eq:sde-Y} with respect to their initial values.

\begin{lem}\label{lemma: cont L0}
 Pick any $(x,y,T) \in \Real_+^3$, and any sequence $\set{(x_n, y_n, T_n)}_{n\in \Natural}$ which converges to $(x,y,T)$. Then,
 \begin{equation}\label{eq: cont L0}
  \plimn Y^{y_n}_{T_n} = Y^y_T \quad \text{ and } \quad \plimn S^{x_n,y_n}_{T_n} =S^{x,y}_T,
\end{equation}
where ``$\plim$'' denotes limit in $\prob$-measure.
\end{lem}

\begin{proof}
 The stability properties of solutions for \eqref{eq:sde-Y} have been well-studied under the linear growth assumption \eqref{eq: mu sigma linear} (see e.g. \cite{Bahlali}). In fact, \eqref{eq: cont L0} follows from Theorem~2.4 in \cite{Bahlali}, which shows
 \begin{equation}\label{eq: Y stability}
  \lim_{n\rightarrow \infty} \expec\bra{\sup_{0\leq u\leq t+\delta} \left|Y^{y_n}_u - Y^y_u\right|^2} = 0, \quad \text{ for any } \delta>0,
 \end{equation}
 and the fact that $\expec\bra{\left|Y^{y}_{t_n} - Y^y_t\right|^2} \leq C(1+y^2) |t-t_n|$ for some $C$ --- see Problem~5.3.15 in \cite{Karatzas-Shreve-BM}.

 For the stability of $S$, it suffices to show that $\plimn \log{H^{y_n}_{t_n} = \log{H^y_t}}$. In the next paragraph, we will prove that
 \begin{equation}\label{eq: stoch int conv}
  \lim_{n\rightarrow \infty} \expec\bra{ \abs{\int_0^t b\pare{Y^y_u} \, dW_u - \int_0^{t_n} b\pare{Y_u^{y_n}} \, dW_u}^2} = 0.
 \end{equation}
 The fact that $\lim_{n\rightarrow \infty} \expec\bra{ \abs{\int_0^t b^2\pare{Y^y_u} \, du - \int_0^{t_n} b^2 \pare{Y_u^{y_n}} \, du}}=0$ can be shown in a similar fashion. Then, $\plimn \log{H^{y_n}_{t_n} = \log{H^y_t}}$ follows from these two identities.

To estimate the left-hand-side of \eqref{eq: stoch int conv}, we use It\^o's isometry to get
 \[
  \expec\bra{\abs{\int_0^t b\pare{Y^y_u} \, dW_u - \int_0^{t_n} b\pare{Y_u^{y_n}} \, dW_u}^2} \leq 2 \, \expec\bra{\int_0^{t_n} \pare{b(Y^y_u) - b(Y^{y_n}_u)}^2\, du} + 2 \, \expec\bra{\left|\int_{t_n}^t b^2(Y^y_u)\, du \right|}.
 \]
 Let $n$ be large enough (greater than or equal to, say, some $N(\delta)$) so that $t_n \leq t+\delta$ and $y_n \leq y+\delta$ for some $\delta>0$.
 Since drift and volatility of $Y^y$ have at most linear growth, it follows  that
 \begin{equation}\label{eq: moment Y}
  \expec [ \sup_{t \in [0, T]} \abs{Y^y_t}^m ] \leq C_{m,T}(1+y^m) \quad \text{ for any } m>0.
 \end{equation}
 On the other hand, \eqref{eq: b poly} implies that $b(y) \leq C(1+y^k)$ for some constants $k$ and $C$. Combining the previous two inequalities with \eqref{eq: moment Y}, we have $\expec\bra{\sup_{u\leq t+\delta} b^2(Y^y_u)} \leq C_{\delta, y}$, for some constant $C_{\delta, y}$. As a result, $\lim_{n\rightarrow \infty}\expec\bra{\left| \int_{t_n}^t b^2(Y^y_u)\, du \right|} \leq  \lim_{n\rightarrow \infty} C_{\delta, y} |t-t_n| =0$. On the other hand, since $b$ is locally H\"{o}lder continuous on $\Real_+$, then for any $M>0$, there exist constants $\alpha\in (0,1]$ and $C_M$ such that $|b(x) - b(y)|^2 \leq C_M |x-y|^{2 \alpha}$ for any $x, y \leq M$. As a result, for any $u\leq t_n$,
 \begin{equation}\label{eq: est b2}
 \begin{split}
  & \expec\bra{\pare{b(Y^y_u) - b(Y^{y_n}_u)}^2}\\
  & \quad = \expec\bra{\pare{b(Y^y_u) - b(Y^{y_n}_u)}^2 \, \indic_{\set{Y^y_u \leq M, \, Y^{y_n}_u \leq M}}} + \expec\bra{\pare{b(Y^y_u) - b(Y^{y_n}_u)}^2 \, \indic_{\set{Y^y_u > M \text{ or } Y^{y_n}_u > M}}}\\
  & \quad \leq C_M \expec\bra{\left|Y^y_u - Y^{y_n}_u\right|^{2\alpha}} + C \,\expec\bra{\pare{2+ \pare{Y^y_u}^{2k} + \pare{Y^{y_n}_u}^{2k}} \, \indic_{\set{Y^y_u > M \text{ or } Y^{y_n}_u > M}}}.
 \end{split}
 \end{equation}
Since $\expec\bra{\left|Y^y_u - Y^{y_n}_u\right|^{2\alpha}} \leq \expec\bra{\left|Y^y_u - Y^{y_n}_u\right|^{2}}^{\alpha}$ holds by Jensen's inequality, it follows from \eqref{eq: Y stability} that the first term on the right-hand-side of \eqref{eq: est b2} converges to zero as $n\rightarrow \infty$. For the second term,
 observe that $\sup_{n \in \Natural} \expec \bra{\pare{Y^{y_n}_u}^{4k}} < \infty$ implies that $\set{\pare{Y^{y_n}_u}^{2k}}_{n \in \Natural}$ is a uniformly integrable family; therefore,
\[
\limsupn \expec\bra{\pare{2+ \pare{Y^y_u}^{2k} + \pare{Y^{y_n}_u}^{2k}} \, \indic_{\set{Y^y_u > M \text{ or } Y^{y_n}_u > M}}} \leq \expec\bra{\pare{2+ 2 \pare{Y^y_u}^{2k}} \, \indic_{\set{Y^y_u \geq M }}}
\]
and the last expression is further dominated by
$
\expec\bra{\pare{2+ 2 \sup_{u \in [0, t + \delta]} \pare{Y^y_u}^{2m}} \, \indic_{\set{ \sup_{u \in [0, t  + \delta]} Y^y_u \geq M }}}.
$
It then follows that
 \[
  \limsupn \int_0^{t_n} \expec\bra{\pare{b(Y^y_u) - b(Y^{y_n}_u)}^2}\, du \leq C(t+\delta)\, \expec\bra{\pare{1+ 2 \sup_{u \in [0, t + \delta]} \pare{Y^y_u}^{2k}} \, \indic_{\set{ \sup_{u \in [0, t  + \delta]} Y^y_u \geq M }}},
 \]
 for some constant $C$.
 Sending $M\rightarrow \infty$, we have that the right-hand-side of last inequality converges $0$ thanks to \eqref{eq: moment Y} and the dominated convergence theorem. This concludes the proof of \eqref{eq: stoch int conv}.
\end{proof}

Now comes the first step towards proving Theorem \ref{thm: existence}.

\begin{lem}\label{thm: u cont}
 $u\in C(\Real_+^3) \cap C^{2,2,1}(\Real_{++}^3)$ and it satisfies \eqref{eq: pricing eq}.
\end{lem}


\begin{proof}
We decompose the proof into three steps. First, we apply regularity results for nondegenerate parabolic PDEs to show that $u$ is continuous in the interior of $\Real_+^3$. Then assuming that $g(x) \equiv x$, we use probabilistic arguments to prove that $u$ extends continuously to the boundaries of $\Real_+^3$. Finally, we generalize the result to general payoff functions.

\subsubsection*{Step 1}
Consider a sequence of payoff functions $g^m := g \wedge m$, for $m \in \Natural$, and define $u^m (x, y, T) := \expec[g^m(S^{x,y}_T)]$ for $(x,y, T) \in \Real_+^3$. The monotone convergence theorem implies that $\lim_{m \rightarrow \infty} u^m(x, y, T) = u(x, y, T)$ for every $(x,y,T)\in \Real_+^3$. For each $u^m$, since $g^m$ is bounded and continuous, the continuity of $u^m$ follows from \eqref{eq: cont L0} and the bounded convergence theorem.

  Now let us consider a cylindrical domain $\D = A \times (t_1, t_2)$ such that its closure $\overline{\D}$ is a bounded subset of $\Real_{++}^3$. Since $\overline{\D}$ avoids the boundaries $x=0$ and $y=0$, it follows from a verification argument (see e.g. Theorem 2.7 in \cite{Janson-Tysk}) that $u^m$ satisfies a uniformly parabolic differential equation $u^m_T = \cL  u^m$ in $\D$. Note that the coefficients of these equations are the same for all $m$ and that $u^m$ are uniformly bounded above by $u$ which is bounded on $\overline{\D}$. It then follows from the \emph{interior Schauder estimate} (see e.g. Theorem~15 in \cite{friedman-parabolic} pp.80) that for any subsequence $\{u^{m'}\}$ of $\{u^m\}$, there exists a further subsequence $\{u^{m''}\}$ such that $\{u^{m''}\}$ uniformly converges to $u$ in any compact subdomain in $\D$. It then follows from the continuity of $u^{m''}$ and the uniform convergence that $u\in C(\D)$. Therefore, $u \in C(\Real_{++}^3)$ since $\D$ is arbitrarily chosen. On the other hand, the Schauder interior estimate also yields that $u$ satisfies \eqref{eq: pricing eq} and $u \in C^{2,2,1}(\Real_{++}^3)$.

\subsubsection*{Step 2} Consider the special case of $g$ satisfying $g(x) \equiv x$; in this case, $u$ satisfies $u(x, y, T) = x\, \expec\bra{H^y_T}$ for $(x, y, T) \in \Real_+^3$. We are going to show that $u$ extends continuously to the boundaries $x=0$, $y=0$, and $T=0$. (If $H^{y}$ is a martingale for $y \in \Real_+$, this step is entirely trivial. Indeed $x \expec\bra{H^{y}_{T}} = x$ clearly indicates that $u$ is continuous on $\Real^3_+$.)

Take an $\Real_+$-valued  sequence $(x_k)_{k \in \Natural}$ such that $\downarrow \lim_{k \to \infty} x_k = 0$. It follows from the supermartingale property of $H^y$ that $|u(x_k, y, T) - u(0, y, T)| = x_k \, \expec\bra{H^y_T} \leq x_k$ for all $(y, T) \in \Real_+^2$. Therefore, $u(x_k, y, T)$ converges uniformly in $(y, T)$ to $u(0, y, T)$. This ensures that $u$ extends continuously to the boundary $x=0$.

 Let us prove the continuity at $T=0$. Given any sequence $\Real_+^3 \ni (x_k, y_k, T_k) \rightarrow (x,y,0)$, it follows from Fatou's lemma and \eqref{eq: cont L0} that $\liminf_{k\rightarrow \infty} u(x_k, y_k, T_k) \geq x\, \expec[\liminf_{k\rightarrow \infty} H^{y_k}_{T_k}] =x$. On the other hand, note that since $\expec[H^{y_k}_{T_k}] \leq 1$ holds for all $k$, $\limsup_{k\rightarrow \infty} u(x_k, y_k, T_k) \leq x$. We then conclude that $u$ extends continuously to $T=0$.

 Since $\lim_{k \rightarrow \infty} u(x_k, y, T) = u(x,y,T)$ uniformly in $(y,T)$, in order to show that $u$ extends continuously to $y=0$, it suffices to show that for any $\Real_+$-valued sequence $\{y_{\ell}\} \downarrow 0$, $\expec\bra{H^{y_\ell}_T}$ converges to $\expec\bra{H^0_T}$ uniformly, and that $\Real_+ \ni T \mapsto \expec\bra{H^0_T}$ is continuous.

 Let us prove the continuity of $\Real_+ \ni T \mapsto \expec\bra{H^0_T}$ first. Recall $\expec\bra{H^0_T} = \prob \bra{\zeta^0 > T}$ from \eqref{eq: expec S}. It is clear that $\Real_+ \ni T \mapsto \prob \bra{\zeta^0 > T}$ is right continuous. In order to show the left continuity of this map, it suffices to show $\prob \bra{\zeta^0 = T} =0$ for any $T \in \Real_+$. To this end, set $\tau= \inf\set{t\geq 0 \such Y^0_t = 1}$. It follows from the strong Markov property that
\begin{equation} \label{eq: H0 cont}
\prob \bra{\zeta^0 = T} = \int_0^{T} \prob \bra{\zeta^1 = T - s} \prob \bra{\tau \in d s}.
\end{equation}
We have shown that $T\mapsto \expec[H_T^1]$ is continuous at $T=0$, moreover we also conclude from Step 1 that the last map is continuous at $T>0$. Therefore,
$\Real_+ \ni T \mapsto \expec\bra{H^{1}_T}$ is continuous, which implies that $\prob \bra{\zeta^1 =t} =0$ for any $t \in \Real_+$. Combining the last fact with \eqref{eq: H0 cont}, we obtain that $\prob \bra{\zeta^0 = T}=0$, which confirms the left continuity of $\Real_+ \ni T \mapsto \expec\bra{H^0_T}$.

 Now we prove $\lim_{\ell\rightarrow \infty} \expec[H^{y_\ell}_T] = \expec[H^0_T]$ for fixed $T$. On one hand, it follows from Fatou's lemma that $\expec\bra{H^0_T} \leq \liminf_{\ell \rightarrow \infty} \expec\bra{H^{y_\ell}_T}$. On the other hand, it follows from Proposition~\ref{prop: mart S} that $\cbra{\expec \bra{H^{y_\ell}_T}}_{\ell \in \Natural}$ is a nondecreasing sequence. This implies that $\limsup_{\ell\rightarrow \infty} \expec[H^{y_\ell}_T] \leq \expec[H^0_T]$. Therefore we have shown $\uparrow \lim_{\ell \rightarrow \infty} \expec[H^{y_\ell}_T] = \expec[H^0_T]$.

To show that the convergence $\uparrow \lim_{\ell \rightarrow \infty} \expec[H^{y_\ell}_T] = \expec[H^0_T]$ is uniform, recall that
$\Real_+ \ni T \mapsto \expec\bra{H^0_T}$ is continuous. On the other hand, $\Real_+ \ni T \mapsto \expec[H^y_T]$ is continuous for $y>0$. 
It then follows from Dini's theorem that the convergence of $\cbra{\expec[H^{y_{\ell}}_T]}_{\ell\in \Natural}$ is uniform in $T$.

\subsubsection*{Step 3} The results of the previous two steps imply that $\Real_+^3 \ni (x,y,T) \mapsto \expec[S^{x,y}_T]$ is continuous on $\Real_+^3$. Hence, for any sequence $\{(x_n, y_n, T_n)\}_{n\in \Natural}$ converging to $(x,y,T)$ with $(x_n, y_n, T_n)$ inside a bounded neighborhood of $(x,y,T)$ for $n\in \Natural$, $\cbra{S_{T_n}^{x_n, y_n}}_{n\in \Natural}$ is a uniformly integrable family. Therefore, for a nonnegative payoff $g$ which is at most linear growth, $\cbra{g(S_{T_n}^{x_n,y_n})}_{n\in \Natural}$ is bounded from above by a uniform integrable family $\cbra{M \pare{1+  S_{T_n}^{x_n,y_n}}}_{n\in \Natural}$, which  along  with \eqref{eq: cont L0} implies that $u\in C(\Real^3_+)$.
\end{proof}

\subsection{Proof of Proposition~\ref{prop: S strict local mart}}
Let $y \in \Real_{++}$. When \eqref{feller test} is violated, it follows from Feller's test that $\prob \bra{\zeta^y = \infty} = 1$. Then, \eqref{eq: expec S} implies that $H^{y}_{\cdot \wedge T}$ is a martingale for any $T \geq 0$. This confirms the implication $(1) \implies (2)$.

 The proof of the implication $(2) \implies (1)$ is motivated by the proof of Proposition 3 in \cite{DFern-Kar09}. Let us define $I(y, T) := \expec [H^y_T] = \prob \bra{\zeta^y > T}$ for $(y,T) \in \Real_+^2$. Since $\expec [S^{x, y}_T] = x \, I(y,T)$, it follows from Lemma~\ref{thm: u cont} (choosing $g$ such that $g(x) \equiv x$) that $I\in C(\Real_+^2) \cap C^{2,1}(\Real_{++}^2)$ and that $I$ satisfies
 \begin{equation}\label{eq: pde I}
 \begin{split}
  & \partial_T I - \frac12 \sigma^2(y) \,\partial^2_{yy} I - (\mu(y) + \rho b(y) \sigma(y)) \,\partial_y I = 0, \quad  (y,T)\in \Real_{++}^2,\\
  & I(y, 0) =1, \quad y\in \Real_+.
 \end{split}
 \end{equation}

 When \eqref{feller test} is satisfied, it follows from Feller's test for explosions that $\lim_{T\rightarrow \infty} I(y, T) < 1$ for all $y \in \Real_{++}$. Pick sufficiently large $T^{*}$ such that $I(1, T^{*}) < 1$. We claim that
 \begin{equation}\label{eq: I<1 all y}
  I(y, T^{*}) <1 \quad \text{ for all } y \in \Real_{++}.
 \end{equation}
 We shall prove this by contradiction. Suppose that there exists  $y^* \in \Real_{++}$ such that $I(y^*, T^{*}) =1$. For any $y>0$, consider an open domain $A$ which contains both $1$ and $y^*$ and whose closure $\overline{A}$ is a compact subset of $\Real_{++}$. Then $I$ attains its maximum at $(y^*, T^{*})$ over the cylindrical domain $A \times [0, T^{*}+1]$. Note that $I$ satisfies the uniformly parabolic equation \eqref{eq: pde I} in $A \times (0, T^{*}+1)$. Then the maximum principle (see e.g.  \cite[Chapter 2]{friedman-parabolic}) implies that $I(y, T) =1$ for any $0\leq T\leq T^{*}$ and $y\in A$. Therefore $I(1, T^*) =1$, which clearly contradicts with the choice of $T^{*}$.

 Now define $\mathcal{S}(T) = \set{y \in \Real_{++} : I(y, T) =1}$ and
 \begin{equation}\label{def: t*}
  T_{*} := \sup\set{T \geq 0 : \mathcal{S}(T) \neq \emptyset},
 \end{equation}
with the convention that $T_{*} = \infty$ when the above set is empty. In fact, \eqref{eq: I<1 all y} implies $T_{*} <\infty$. We shall show $T_{*} = 0$ in what follows.

 Suppose $T_{*} >0$. Then for any $\delta \in (0, T_{*}/2)$, there exists a $y \in \Real_{++}$ such that $I(y, T_{*} - \delta) =1$. Using the maximum principle as we did above, we obtain that
 \begin{equation}\label{eq: I=1 all y t}
  I(y, T) = 1, \quad \text{ for any } 0\leq T\leq T_{*} - \delta \text{ and } y \in \Real_+.
 \end{equation}
(Note that  $I(0, T)=1$ follows because $I(\cdot,T)$ is nonincreasing for fixed $T \in \Real_+$  --- see Proposition \ref{prop: mart S}.) Now, from the definition of $I$ and the Markov property, we have $\expec \bra{H_T^y \, |\, \F_t} = I(Y^y_t, T-t)$ for all $(y, T) \in \Real_{++}^2$. When $0\leq t\leq T_* - \delta$ and $0\leq T-t\leq T_* - \delta$, applying \eqref{eq: I=1 all y t} to the previous identity, we obtain $I(y, T) = 1$ for every $T\in [0, 2(T_{*}-\delta)]$ and $y \in \Real_{++}$. Note that $2(T_{*} - \delta) > T_{*}$, this contradicts with the definition of $T_{*}$.
 Therefore, $T_{*} = 0$, which implies that $I(y,T) < 1$ for any $(y, T) \in \Real_{++}^2$. \qed

\section{The Notion of Stochastic Solutions}\label{sec: stoch soln}
A notion of stochastic solutions to \eqref{eq: pricing eq} is introduced in this section. Its definition is motivated by the definition in \cite{Stroock-Varadhan} pp. 672, Definition~3.1 in \cite{Hsu}, and Definition~2.2 in \cite{Janson-Tysk}. 


\begin{defn}\label{defn: stochastic solution}
Consider a continuous function $v: \Real_+^3 \rightarrow \Real_+$. For $(x, y, T) \in \Real_{++}^3$, define $V^{x, y, T} = (V^{x, y, T}_t)_{t \in [0, T]}$ via $V^{x, y, T}_t = v(S^{x,y}_t, Y^{y}_t, T-t)$ for $t \in [0, T]$. Then, $v$ is a \emph{stochastic solution} of \eqref{eq: pricing eq}, if for each $(x,y,T) \in \Real_{++}^3$:
 \begin{enumerate}
  \item[(i)] $V^{x, y, T}$ is a local martingale on $[0, T]$,
  \item[(ii)] $v\pare{x, y, 0} = g\pare{x}$.
 \end{enumerate}
\end{defn}


\begin{prop}\label{prop: existence}
 The value function $u$, defined in \eqref{eq: def u}, is a stochastic solution dominated by $h$. In fact, $u$ is the smallest stochastic solution.
\end{prop}

\begin{proof}
We have already shown in Lemma~\ref{thm: u cont} that $u\in C(\Real_+^3)$ and $u\leq h$ in \eqref{eq: u growth}. Recall that $U^{x, y, T} = (U^{x, y, T}_t)_{t \in [0, T]}$ with $U^{x, y, T}_t =u(S^{x,y}_t, Y^{y}_t, T-t)$ for $t \in [0, T]$. In \eqref{eq: mart u} we established that $U^{x, y, T}$ is a martingale on $[0,T]$. Therefore, $u$ is a stochastic solution.


To show the second statement, we take another stochastic solution $v$ and let $V^{x, y, T}$ be as in Definition \ref{defn: stochastic solution}. Since $V^{x, y, T}$ is a nonnegative local martingale, hence a supermartingale, we have
\[
v(x,y,T) = V^{x, y, T}_0 \geq \expec \bra{V^{x, y, T}_T} = \expec \bra{v\pare{S^{x, y}_{T}, Y^y_{T}, 0}} = \expec \bra{g\pare{S^{x, y}_{T}}} = u(x,y,T).
\]
Therefore $v\geq u$ on $\Real_{++}^3$. Thanks to the continuity of $v$ and $u$ on $\Real_+^3$, the last inequality then holds on $\Real_+^3$.
\end{proof}

The uniqueness of stochastic solutions for \eqref{eq: pricing eq} ties naturally to the martingale property of the  asset-price process. This result is the main accomplishment of this section which will be presented in two propositions. But before we will need to state the following technical lemma.

Before proceed, let us prepare the following result.

\begin{lem}\label{lem: h growth}
 $\eta = \limsup_{x\rightarrow \infty} g(x)/x = \limsup_{x\rightarrow \infty} h(x)/x= \downarrow\lim_{x\rightarrow \infty} h'(x)$.
\end{lem}

\begin{proof}
 Since $h$ dominates $g$,
 \begin{equation}\label{eq: h(x)/x lb}
 \limsup_{x\to \infty} \frac{h(x)}{x} \geq \limsup_{x\to \infty} \frac{g(x)}{x} = \eta.
 \end{equation}
 If $\downarrow \lim_{x\to \infty} h'(x)<\eta$, there exist $x_0$ and $\epsilon>0$ such that $h'(x)\leq \eta-\epsilon$ for $x\geq x_0$. Hence $h(x) \leq (\eta- \epsilon)(x-x_0) + h(x_0)$, for $x\geq x_0$, which contradicts \eqref{eq: h(x)/x lb}. On the other hand, if $\downarrow \lim_{x\to \infty} h'(x) =\xi > \eta$, there exists $x_0$, such that $h(x_0)\geq \frac{\xi+\eta}{2} x_0$ and $g(x) \leq \frac{\xi+ \eta}{2} x$ for $x>x_0$. Since $h'(x)\geq \xi$ on $\Real_+$, $h(x)>h(x_0)-\frac{\xi+ \eta}{2} x_0+\frac{\xi+ \eta}{2} x$, for $x>x_0$. Let us consider
 \[
  \wt{h}(x) := \left\{\begin{array}{ll}h(x), & x< x_0;\\ h(x_0)-\frac{\xi+ \eta}{2} x_0+\frac{\xi + \eta}{2} x, & x\geq x_0.\end{array}\right.
 \]
 It is easy to check that $\wt{h}$ is another nonnegative, nondecreasing, and concave function that dominates $g$. But $\wt{h}<h$, which contradicts with the definition of $h$. Therefore, $\downarrow \lim_{x\to \infty} h'(x) =\eta$.

To show $\limsup_{x\to \infty} h(x)/x = \eta$, observe  that
\[
\limsup_{x\to \infty} \frac{h(x)}{x} = \limsup_{x\to \infty} \frac{h(x)- h(0)}{x} \geq \lim_{x\rightarrow \infty} h'(x) = \eta,
\]
since $h$ is concave. On the other hand, for any $\epsilon >0$, there exists $x_0$, such that $h'(x) \leq \eta+ \epsilon$ for $x\geq x_0$. Therefore $h(x) \leq (\eta+\epsilon)(x-x_0) + h(x_0)$ for $x\geq x_0$, which implies $\limsup_{x\to \infty} h(x)/ x \leq \eta+ \epsilon$. Hence $\limsup_{x\to \infty} h(x) /x \leq \eta$ since the choice of $\epsilon$ is arbitrary.
\end{proof}

\begin{prop}\label{prop: uniqueness}
Suppose that $g$ is of linear growth, i.e., $\eta = \limsup_{x\rightarrow \infty} g(x) / x>0$. Then, there exists a unique stochastic solution in the class of functions which are dominated by $h$ if and only if the asset-price process is a martingale. In that case, $u$ is this unique stochastic solution.
\end{prop}

\begin{proof}
 Let us define a function $\delta: \Real_+^3 \mapsto \Real$ via $\delta (x, y, T) := x- \expec \bra{S^{x, y}_T} = x- x\,\expec \bra{H^y_T}$ for $(x, y, T) \in \Real_+^3$. Since $H^y$ is a nonnegative local martingale for $y \in \Real_+$, $\delta$ is nonnegative. Also,
\[
\delta \pare{S^{x, y}_{t}, Y^y_{t}, T- t} = S^{x, y}_{t} - \expec \bra{S^{x, y}_T \such \F_{t} }
\]
holds for all $(x, y, T) \in \Real_+^3$ and $t \in [0, T]$, in view of the Markov property. It follows that $(\delta \pare{S^{x, y}_{t}, Y^y_{t}, T- t})_{t \in [0,T]}$ is a local martingale for all $(x, y, T) \in \Real_+^3$.

Now, Lemma~\ref{lem: h growth} implies that $f(x):=h(x) - \eta x$ is a nondecreasing concave function. Hence
\begin{equation}\label{eq: u+eta delta}
\begin{split}
 (u+ \eta \delta) (x,y,T) &= \expec\bra{g(S^{x,y}_T) - \eta S^{x,y}_T} + \eta x \leq \expec\bra{f(S_T^{x,y})} + \eta x\\
 & \leq f\pare{\expec[S^{x,y}_T]} + \eta x \leq f(x) + \eta x = h(x), \quad \text{ for any } (x,y,T)\in \Real_+^3.
\end{split}
\end{equation}
Therefore, both $u$ and $u+ \eta \delta$ are stochastic solutions dominated by $h$. Suppose that the stochastic solution is unique. Then the  asset price process must be a martingale. Otherwise, Proposition~\ref{prop: S strict local mart} implies that $S^{x,y}_{\cdot \wedge T}$ is a strict local martingale for any $(x,y,T)\in \Real_{++}^3$; hence, $\delta>0$ on $\Real_{++}^3$ and $u$ and $u+\eta \delta$ are two different stochastic solutions dominated by $h$.

Assume that the  asset-price process is a martingale and take a stochastic solution $v$ which is dominated by $h$ on $\Real^3_+$. The uniqueness follows once we show $v \equiv u$. We shall establish below that $v = u$ on $\Real_{++}^3$. The last identity can be extended to $\Real_{+}^3$ thanks to the continuity of $v$ and $u$.

Fix $(x,y,T)\in \Real_{++}^3$, and take a localizing sequence $\set{\sigma_n}_{n\in \Natural}$ of the local martingale $V^{x,y, T}$. Then,
 \[
  v(x,y,T) = V^{x,y, T}_0 = \expec\bra{V^{x,y,T}_{\sigma_n \wedge T}} \quad \text{ for all } n\in \Natural.
 \]
On the other hand, the linear growth constraint $h(x)\leq M(1+x)$ on $\Real_+$ implies that
\[
 V^{x,y,T}_{\sigma_n \wedge T} \leq M  \pare{1+ x H^y_{\sigma_n \wedge T}}.
\]
Since $H^y$ is a martingale, $\set{H^y_{\sigma_n\wedge T}}_{n\in \Natural}$ is a uniformly integrable family. Therefore, $\set{V^{x,y,T}_{\sigma_n \wedge T}}_{n\in \Natural}$ is a uniformly integrable family, which along with the continuity of $v$ implies that
 \[
 \begin{split}
  v(x,y,T) = \lim_{n\rightarrow \infty} \expec \bra{V^{x,y,T}_{\sigma_n \wedge T}}=  \expec \bra{\lim_{n\rightarrow \infty} V^{x,y,T}_{\sigma_n \wedge T}} = \expec \bra{V^{x,y,T}_{T}} = \expec \bra{g \pare{S^{x, y}_{T}}} = u(x,y,T).
 \end{split}
 \]
\end{proof}

When the payoff $g$ is of strictly sublinear growth, the uniqueness of stochastic solutions always holds, no matter whether the  asset-price process is a martingale or not.

\begin{prop}\label{prop: strtictly sublinear}
 When $g$ be of strictly sublinear growth, i.e., $\limsup_{x\rightarrow \infty} g(x)/x=0$, then $u$ is the unique stochastic solution dominated by $h$.
\end{prop}

\begin{proof}
Fix $T \in \Real_+$. It follows from Lemma~\ref{lem: h growth} that $\lim_{x\rightarrow \infty} h(x)/x = 0$. Then there exists a nondecreasing function $\phi:\Real_+ \rightarrow \Real_+ \cup\set{\infty}$ with $\lim_{x\rightarrow \infty} \phi(x) / x = \infty$, such that $\phi(h(x)) \leq x$ holds for all $x\in \Real_+$. Therefore, for any localizing sequence $\set{\sigma_n}_{n\in \Natural}$ of the local martingale $V^{x,y, T}$, we have
 \[
  \expec \bra{\phi\pare{h \pare{S^{x, y}_{\sigma_n \wedge T}}}} \leq \expec \bra{S^{x, y}_{\sigma_n \wedge T}} \leq x, \quad \text{ for all } n \in \Natural.
 \]
 From de la Vall\'{e}e-Poussin criterion, $\set{h\pare{S^{x, y}_{\sigma_n \wedge T}}}_{n\in \Natural}$ is a uniformly integrable family. The rest follows from arguments similar to the ones used in the proof of Proposition~\ref{prop: uniqueness}.
\end{proof}

\section{Proofs of Main Results}\label{sec: proof of main thm}


The proof consists of three steps. First, the value function is shown to be a classical solution to \eqref{eq: pricing eq} in Section~\ref{sec: u classical soln}. Second, any classical solution is proved to be a stochastic solution in Section~\ref{sec: clasical->stoch}. Finally, in Section~\ref{sec:fsb}, Theorems~\ref{thm: existence} and \ref{thm: uniqueness} are proved utilizing the results of Section~\ref{sec: stoch soln}.

\subsection{The value function is a classical solution}\label{sec: u classical soln}
Let us first focus on Case (C) in Definition~\ref{def: classical soln}.

\begin{lem}\label{lem: u_eps}
 In Case (C) of Definition~\ref{def: classical soln}, $u\in \overline{\fC}$.
\end{lem}
\begin{proof}
  Since $g$ satisfies Assumption~\ref{ass: payoff}, there exists a sequence $\set{g^\epsilon}_{\epsilon>0}$, such that, for each $\epsilon$:
 \begin{enumerate}
	\item[(i)] $g^{\epsilon}$ is bounded;
	\item[(ii)] $g^{\epsilon} \in C^{\infty}(\Real_{++})$;
	\item[(iii)] $(g^{\epsilon})^{'}$ and $(g^{\epsilon})^{''}$ have compact support in $\Real_{++}$;
	\item[(iv)] $g^{\epsilon}(x) \leq h(x)$ for $x\in \Real_+$, and
	\item[(v)] $\lim_{\epsilon\downarrow 0} g^{\epsilon}(x) = g(x)$ for $x\in \Real_+$.
\end{enumerate}
Indeed, for $\epsilon \in (0,1)$, consider
\[
 \wt{g}^\epsilon(x) = \left\{\begin{array}{ll}g(0), & -2\epsilon \leq x \leq 2 \epsilon \\ g(x) \psi^\epsilon(x), & x > 2\epsilon \end{array}\right., \quad \text{ where } \psi^\epsilon\in C(\Real_+) \text{ and } \psi^\epsilon(x) = \left\{\begin{array}{ll}1, & x\leq 1/\epsilon \\ 0, & x>2/\epsilon \end{array}\right..
\]
Then define $g^\epsilon := \eta^\epsilon  \ast \wt{g}^\epsilon$, $\eta^{\epsilon}$ is the standard mollifier and $\ast$ denotes the convolution operator. It is clear that $g^\epsilon(x) = g(0)$ for $x\in [0, \epsilon]$. Therefore, items (i) - (iii) and (v) are clearly satisfied. In order to check item (iv), we notice that $g^\epsilon(x) = g(0)\leq h(0) \leq h(x)$ for $x\in [0, \epsilon]$. On the other hand, we claim $\int_{-\epsilon}^{\epsilon} \eta^\epsilon(y) h(x-y)dy \leq h(x)$ for $x>\epsilon$. This claim follows from $\eta^\epsilon(y) = \eta^{\epsilon}(-y)$, $\int_{-\epsilon}^{\epsilon} \eta^\epsilon(y) dy =1$, and $h(x+y)- h(x) \leq h(x) - h(x-y)$ for $y>0$ thanks to the concavity of $h$. Hence item (iv) holds because $g^{\epsilon}(x) = \int_{-\epsilon}^{\epsilon} \eta^{\epsilon}(y) \tilde{g}^{\epsilon}(x-y) dy \leq \int_{-\epsilon}^{\epsilon}\eta^{\epsilon}(y) h(x-y) dy \leq h(x)$, for $x> \epsilon$.

Define $u^{\epsilon}(x,y,T) := \expec[g^{\epsilon}(S^{x,y}_T)]$ for $(x,y,T)\in \Real_+^3$. An estimate similar to \eqref{eq: u growth} shows $u^\epsilon \leq h$ on $\Real_+^3$. Moreover, item (iv) above and the dominated convergence theorem combined implies that
\[
 u(x,y,T) = \lim_{\epsilon \downarrow 0} u^{\epsilon}(x,y,T), \quad \text{ for } (x,y,T)\in \Real_+^3.
\]
Then the statement follows, if $u^{\epsilon}\in \fC$ for each $\epsilon\in (0,1)$. This property of $u^\epsilon$ will be confirmed in the rest of the proof using an argument from \cite{Ekstrom-Tysk-stochvol}.

First, boundedness of $g^\epsilon$ and \eqref{eq: cont L0} combined implies that $u^\epsilon \in C(\Real_+^3)$. Then an argument similar to that in Lemma~\ref{thm: u cont} shows that $\partial_T u^{\epsilon} = \cL u^{\epsilon}$ on $\Real^3_{++}$. Moreover, the dominated convergence theorem implies that
 \[
  x^2\, \partial^2_{xx} u^{\epsilon}(x,y,T) = \expec \bra{(S^{x,y}_T)^2 \,(g^{\epsilon})^{''}(S^{x,y}_T)}.
 \]
 Since $(g^{\epsilon})^{''}$ has compact support and it is finite at $x=0$, $x^2\, |\partial^2_{xx} u^{\epsilon}|$ is bounded on $\Real^3_+$, hence $\lim_{y\downarrow 0} b^2(y) |\partial^2_{xx} u^{\epsilon}(x,y,T)|=0$ for $(x,T)\in \Real_{++}^2$.

The proof that $\partial_T u^\epsilon, \partial_y u^\epsilon \in C(\Real_{++} \times \Real_+ \times \Real_{++})$ follows the line of arguments presented in \cite{Ekstrom-Tysk-stochvol}. The assumptions on the payoff and coefficients in \cite{Ekstrom-Tysk-stochvol} are satisfied in our case (see properties of $g^\epsilon$ in items (i) - (iii)).
 Even though $b(y)$ is chosen as $\sqrt{y}$ in \cite{Ekstrom-Tysk-stochvol}, their arguments go through if $b^2 \in C^1(\Real_+)$, $(b^2)^{'}$ H\"{o}lder continuous, and has at most polynomial growth. In particular, (22) in \cite{Ekstrom-Tysk-stochvol} is replaced by $1\leq b^2(\frac{\overline{y}}{m}) x_0^2 \frac{k^2}{m} \leq 2$. For a sequence $\set{m_n}_{n\in \Natural} \uparrow \infty$, a sequence $\set{k_n}_{n\in \Natural}$ can still be chosen appropriately so that above inequalities are satisfied. Moreover, Proposition 4.1 of \cite{Ekstrom-Tysk-stochvol} still holds. Indeed, for any $(x,y)\in \Real_+^2$ and a sequence $\set{(x_n, y_n)}_{n\in \Natural}$ in a bounded neighborhood of $(x,y)$, there exists a constant $C_{T, \epsilon}$ such that
 \[
  \left|\exp \pare{\int_0^{\nu} \mu^{'} (Y^{y_n}_\sigma)\,d \sigma} \pare{b^2(Y^{y_n}_{\nu})}^{'} \pare{S^{x_n, y_n}_{\nu}}^2 \partial^2_{xx} u^{\epsilon} \pare{S^{x_n, y_n}_{\nu}, Y^{y_n}_{\nu}, T-\nu}\right| \leq C_{T, \epsilon} \left|\pare{b^2(Y^{y_n}_{\nu})}^{'}\right|,
 \]
 for any $n\in \Natural$ and $\nu \in [0,T]$. Thanks to the growth assumption on $(b^2)^{'}$ and the moment estimate \eqref{eq: moment Y},
 $\set{\left|\pare{b^2(Y^{y_n}_{\nu})}^{'}\right|}_{n\in \Natural}$ is a uniformly integrable family. Therefore, the function $v^{\epsilon}$ defined as
 \[
  v^{\epsilon}(x,y,T):= \expec\bra{\int_0^T \exp \pare{\int_0^{\nu} \mu^{'} (Y^{y}_\sigma)\,d \sigma} \pare{b^2(Y^{y}_{\nu})}^{'} \pare{S^{x, y}_{\nu}}^2 \partial^2_{xx} u^{\epsilon} \pare{S^{x, y}_{\nu}, Y^{y}_{\nu}, T-\nu}\, d\nu},
 \]
 is still a continuous function on $\Real^3_+$.
\end{proof}

\begin{rem}\label{rem: first order bc}
 It is also proved in \cite{Ekstrom-Tysk-stochvol} that $u^\epsilon$ satisfies
 \begin{equation}\label{eq: boundary cond 1}
  \partial_T u^\epsilon u(x,0,T) = \mu(0) \partial_y u^\epsilon(x,0,T), \quad (x,T)\in \Real^2_{++}.
 \end{equation}
 This first order equation will not be used to prove the uniqueness of classical solutions in Theorem~\ref{thm: uniqueness}. Therefore, according to the consideration in Remark~\ref{rem: boundary cond}, \eqref{eq: boundary cond 1} is not included in the definition of the classical solution as a boundary condition at $y=0$.
\end{rem}

\begin{prop}\label{prop: u classical soln}
The value function $u$ is a classical solution to \eqref{eq: pricing eq}.
\end{prop}
\begin{proof}
 When $\prob[\tau_0^y =\infty]=1$ or $\prob[\tau_0^y <\infty]>0$ with $\mu(0)=0$, the statement follows from Lemma~\ref{thm: u cont} and the fact that $u$ satisfies \eqref{eq: boundary cond 2} when the boundary point $0$ is absorbing. When $\prob[\tau_0^y<\infty]>0$ and $\mu(0)>0$, the statement follows from Lemmas~\ref{lem: u_eps} and \ref{thm: u cont}.
\end{proof}

Recall that $\delta(x,y,T) = x- \expec[S^{x,y}_T]$. The following result follows from Proposition~\ref{prop: u classical soln} when $g(x) \equiv x$.
\begin{cor}\label{cor: delta}
 $\delta$ is a classical solution to \eqref{eq: pricing eq} with zero initial condition.
\end{cor}

\subsection{Any classical solution is a stochastic solution}\label{sec: clasical->stoch}
In order to connect results in the last section to the main results, classical solutions are shown to be stochastic solutions in this section. To facilitate our analysis on Case (C), let us first study the probabilistic property of functions in the class $\fC$.
\begin{lem}\label{lemma: class C}
 For any $v\in \fC$ and $n\in \Natural$, $V^{x,y,T}_{\cdot \wedge \sigma_n} := v\pare{S^{x,y}_{\cdot \wedge \sigma_n}, Y^y_{\cdot \wedge \sigma_n}, T- \cdot \wedge \sigma_n}$ is a martingale on $[0,T]$. Here, $\sigma_n := \inf\set{t\in \Real_+ \such Y_t=n, \text{ or } S_t = n^{-1}, \text{ or } S_t =n} \wedge (T-T/n)$, for $n\in \Natural$.
\end{lem}

Before proving this result, we will analyze the properties of the local time for $Y$. Let $L_t(\epsilon)$ denote the local time $Y$ accumulates at level $\epsilon$ up to time $t\in \Real_+$. Recall that we choose $L$ to be $\prob$-a.s. jointly continuous in the time variable and \cadlag\, in the spatial variable --- see Theorem~3.7.1 of \cite{Karatzas-Shreve-BM}. The following two results will be useful in the proof of Lemma~\ref{lemma: class C}.

\begin{lem}\label{lemma: local time Y}
Fix $y \in \Real_+$. If $\mu(0)>0$, then $L_t(0)=0$ for all $t\geq 0$. Hence $\int_{\Real_+} \indic_{\set{Y^y_t = 0}} d t= 0$.
\end{lem}

\begin{rem}
As the proof below suggests, for the validity of Lemma \ref{lemma: local time Y} we only use that $\sigma$ is locally $(1/2)$-H\"{o}lder continuous on $\Real_+$, it is strictly positive on $\Real_{++}$, and it satisfies $\sigma(0)=0$.
\end{rem}

\begin{proof}
We fix $y \in \Real_+$ and drop superscripts $y$ from $Y^y$ for the ease of notation.
Since $\langle Y, Y \rangle = \int_0^\cdot \sigma^2(Y_t) \,dt$, it follows from the occupation time formula (see e.g. Theorem 3.7.1 (iii) in \cite{Karatzas-Shreve-BM}) that
 \begin{equation}\label{eq: local time est}
 \begin{split}
  t \geq \int_0^t \indic_{(0,\infty)} (Y_u)\,du = \int_0^t \indic_{(0,\infty)} (Y_u) \,\sigma^{-2}(Y_u) \,d \langle Y,Y\rangle_u = 2 \int_{(0, \infty)} \sigma^{-2}(a) \, L_t(a) \, da,
 \end{split}
 \end{equation}
 in which the first equality follows since $\sigma(y)>0$ for $y>0$. Since $\sigma(0)=0$ and $\sigma$ is $(1/2)$-H\"{o}lder continuous in a neighborhood of $0$, we have that $\sigma(a) \leq C a^{1/2}$ for $a \in [0, a_0]$, where $C$ and $a_0$ are $\Real_{++}$-valued constants. Hence, $\sigma^{-2}$ is not integrable in this neighborhood of $0$. Combining the last fact with the \cadlag property of $L$ in the spatial variable, it can be seen that if $L_t(0)$ were not zero, the right-hand-side of \eqref{eq: local time est} would be equal to infinity. This, however, contradicts with the bound on the leftmost-side of \eqref{eq: local time est}. It then follows from Problem 3.7.6 in \cite{Karatzas-Shreve-BM} and $L_t(0) = 0 = L_t(0-)$ that
 \[
 0 = L_t(0) - L_t(0-) = \mu(0) \int_0^t \indic_{\set{Y_u=0}} \, du.
 \]
Since $\mu(0)>0$,  the result follows.
\end{proof}

\begin{lem} \label{lem: loc time UI}
\[\sup_{\epsilon \in (0,1)} \expec \bra{
\pare{L_{\sigma_n} (\epsilon)}^2} < \infty.\]
\end{lem}

\begin{proof}
Let $C \dfn \sup_{y \in [0,n]} \pare{|\mu(y)| + \sigma^2(y)} <
\infty$. From the It\^o-Tanaka-Meyer formula, we obtain
\begin{align*}
L_{\sigma_n}(\epsilon) &\leq \max \set{Y^y_{\sigma_n} - \epsilon , 0} -
\int_0^{\sigma_n} \indic_{(\epsilon, \infty)} (Y^y_t) \mu(Y^y_t) dt  -
\int_0^{\sigma_n} \indic_{(\epsilon, \infty)} (Y^y_t)
\sigma(Y^y_t) dB_t \\
& \leq n + C T - \int_0^{\sigma_n} \indic_{(\epsilon, \infty)} (Y^y_t)
\sigma(Y^y_t) dB_t.
\end{align*}
Furthermore, we have from It\^{o} isometry that
\[
\expec \bra{ \abs{\int_0^{\sigma_n} \indic_{(\epsilon, \infty)} (Y^y_t)
\sigma(Y^y_t) dB_t}^2 } \leq  \expec \bra{ \int_0^{T}
\indic_{(\epsilon, \infty)} (Y^y_t)
\sigma^2(Y^y_t) dt}  \leq   C T.
\]
Combining the last two bounds, we conclude that $\sup_{\epsilon \in
(0,1)} \expec \bra{ \pare{L_{\sigma_n} (\epsilon)}^2} < \infty$.
\end{proof}

\begin{proof}[Proof of Lemma~\ref{lemma: class C}]
 In the sequel, we fix $(x,y,T)\in \Real_{++}$ and drop all superscripts involving $x$, $y$ and $T$ in order to ease notation. Since $v\in C^{2,2,1}(\Real_{++}^3)$ but $Y$ hits zero with positive probability in this case, one cannot directly apply It\^{o}'s Lemma to $V_t$ for $t>\tau_0$. Instead, we apply It\^{o}'s formula to a sequence of processes that approximate $V$.

For $\epsilon \in (0,1]$, define $\ey \dfn \max \set{Y, \epsilon}$.
It follows from the It\^o-Tanaka-Meyer formula that
\[
 d \ey_t = \indic_{(\epsilon, \infty)} (Y_t) \pare{\mu(Y_t) dt +
\sigma(Y_t) dB_t} + d L_t(\epsilon).
\]
Let $\ev$ be defined via $\ev_t := v(S_t, \ey_t, T-t)$ for $t \in [0, T]$.
Since $v \in C^{2,2,1}(\Real_{++}^3)$ and $(S, \ey)$ takes values in
$[n^{-1}, n] \times [\epsilon, n]$ for $t \in [0, \sigma_n]$, we can
apply It\^{o}'s formula on $t \in [0, \sigma_n]$ and obtain
\begin{equation}\label{eq:ito}
\begin{split}
 & \ev_{\cdot \wedge \sigma_n} = v(x, y, T) - \int_0^{\cdot \wedge
\sigma_n} \partial_T v(S_u, \ey_u, T-u) du + \int_0^{\cdot \wedge
\sigma_n} \partial_x v(S_u, \ey_u, T-u) S_u b(Y_u) \, dW_u\\
 & \hspace{.3cm} + \int_0^{\cdot \wedge \sigma_n}  \indic_{(\epsilon,
\infty)} (Y_u) \partial_y v(S_u, \ey_u, T-u) \mu(Y_u) \, du +
\int_0^{\cdot \wedge \sigma_n} \indic_{(\epsilon, \infty)} (Y_u)
\partial_y v(S_u, \ey_u, T-u) \sigma(Y_u) \, dB_u \\
 & \hspace{.3cm} + \int_0^{\cdot \wedge \sigma_n} \partial_y v(S_u,
\ey_u, T-u) \, dL_u(\epsilon) + \frac12 \int_0^{\cdot \wedge \sigma_n}
\indic_{(\epsilon, \infty)} (Y_u) \partial^2_{yy} v(S_u, \ey_u, T-u)
\sigma^2(Y_u) \,du \\
 & \hspace{.3cm} + \int_0^{\cdot \wedge \sigma_n} \indic_{(\epsilon,
\infty)} (Y_u) \partial^2_{xy} v(S_u, \ey_u, T-u) \rho \sigma(Y_u)
b(Y_u) S_u \, du \\
 & \hspace{.3cm} +\frac12 \int_0^{\cdot \wedge \sigma_n}
\partial^2_{xx} v(S_u, \ey_u, T-u) S_u^2 b^2(Y_u)\, du.
\end{split}
\end{equation}
Since $\set{Y > \epsilon} \subseteq \set{\ey = Y}$, it follows from
(BS-PDE) that
\begin{equation}\label{eq: interior-eq}
\begin{split}
&\int_0^{\cdot \wedge \sigma_n} \indic_{(\epsilon, \infty)} (Y_u)
\bra{(\partial_T - \cL) v(S_u, Y_u, T-u)} \, du =0.
\end{split}
\end{equation}
On the other hand, $\int_0^{\cdot \wedge \sigma_n} \partial_y v(S_u,
\ey_u, T-u) \, dL_u(\epsilon) = \int_0^{\cdot \wedge \sigma_n}
\partial_y v(S_u, \epsilon, T-u) \, dL_u(\epsilon)$, following from
the fact that $\int_0^{\cdot \wedge \sigma_n} \indic_{\{Y_u \neq
\epsilon\}} dL_u(\epsilon) =0$. Moreover, the two stochastic integrals in \eqref{eq:ito} are martingales thanks to the choice of $\sigma_n$. As a result, combining \eqref{eq:ito} and \eqref{eq: interior-eq}, and setting
\begin{align*}
\eM \dfn & \ev + \int_0^{\cdot} \indic_{[0,\epsilon]}(Y_u) \partial_T
v(S_u, \epsilon, T-u) du - \int_0^{\cdot} \partial_y v(S_u, \epsilon,
T-u) \, dL_u(\epsilon) \\
& - \frac12 \int_0^{\cdot} \indic_{[0, \epsilon]}(Y_u) \partial^2_{xx}
v(S_u, \epsilon, T-u) S_u^2 b^2(Y_u)\, du,
\end{align*}
we have that $\eM_{\cdot \wedge \sigma_n}$ is a martingale for each $\epsilon\in (0,1)$.

Next we shall study the limit of $\eM$ as $\epsilon\downarrow 0$ and establish
\begin{equation}\label{eq: p conv eps M}
\plim_{\epsilon \downarrow 0} \sup_{t \in [0, T]} \abs{\eM_{t \wedge
\sigma_n} - V_{t \wedge \sigma_n}} = 0.
\end{equation}
First observe that
\begin{equation}\label{eq: conv est}
\begin{split}
 & \sup_{t \in [0, T]} \abs{\eM_{t\wedge \sigma_n} - V_{t \wedge \sigma_n}} \\
 & \hspace{0.3cm} \leq \sup_{t \in [0, T]} \abs{\ev_{t\wedge \sigma_n} - V_{t\wedge \sigma_n}} + \sup_{t \in [0, T]} \int_0^{t\wedge \sigma_n} \indic_{[0,\epsilon]}(Y_u) \abs{\partial_T v(S_u, \epsilon, T-u)} du \\
 & \hspace{0.7cm} + \sup_{t \in [0, T]} \int_0^{t\wedge \sigma_n} \abs{\partial_y v(S_u, \epsilon,
T-u)} \, dL_u(\epsilon)  \\
 & \hspace{0.7cm} +  \frac12 \sup_{t \in [0, T]} \int_0^{t\wedge \sigma_n} \indic_{[0, \epsilon]}(Y_u) \abs{\partial^2_{xx}
v(S_u, \epsilon, T-u)} S_u^2 b^2(Y_u)\, du,
\end{split}
\end{equation}
We will show that each term on the right-hand-side of the previous inequality converges to zero in probability as $\epsilon\downarrow 0$. Let us denote $\D_n=[n^{-1}, n]\times [0,n]\times [T/n,T]$.
First, the convergence of the first term follows from the continuity of $v$. Second, since $\partial_T v\in C(\Real_{++} \times \Real_+ \times \Real_{++})$, we have from Lemma~\ref{lemma: local time Y} that
\[
 \plim_{\epsilon\downarrow 0} \sup_{t \in [0, T]} \int_0^{t\wedge \sigma_n} \indic_{[0,\epsilon]}(Y_u) \abs{\partial_T v(S_u, \epsilon, T-u)} du \leq \sup_{(x,y,s)\in \D_n} \abs{\partial_T v(x,y,s)} \int_0^{T\wedge \sigma_n} \indic_{\set{Y_u=0}} \,du =0.
\]
Since $\limsup_{y\downarrow 0} b^2(y)x^2 \partial^2_{xx} v(x,y,t) <\infty$ for $(x,t)\in [n^{-1}, n]$ and $\partial^2_{xx} v$ is continuous in the interior of $\D_n$, then an argument similar to the previous estimate shows that the fourth term in \eqref{eq: conv est} also converges to zero. Finally, using Lemma~\ref{lemma: local time Y} again, we have the following estimate for the third term,
\[
 \plim_{\epsilon \downarrow 0} \sup_{t \in [0, T]} \int_0^{t\wedge \sigma_n} \abs{\partial_y v(S_u, \epsilon,
T-u)} \, dL_u(\epsilon) \leq \sup_{(x,y,s)\in \D_n} \abs{\partial_y v(x,y,s)} \cdot \plim_{\epsilon \downarrow 0} L_{T \wedge \sigma_n}(\epsilon) =0,
\]
where the last identity follows from the right-continuity of $\epsilon \mapsto L_\cdot(\epsilon)$ and $L (0) = \mu(0) \int_0^{\cdot} \indic_{\set{Y_u = 0}} du$ --- see Theorem 3.7.1 (iv) and Problem 3.7.6 in \cite{Karatzas-Shreve-BM}. As a result, \eqref{eq: p conv eps M} follows combining all the previous estimates.

To finish the proof, we shall show that $V_{\cdot \wedge \sigma_n}$ is a martingale. Using again the facts that $v\in \fC$ and $(S, \ey)$ takes values in $[n^{-1}, n] \times [\epsilon, n]$, we obtain the existence of $C\in \Real_{++}$ (depending on $v$ as well as $n$ but independent of $\epsilon$), such that
\[
 \sup_{t\in [0,T]}\abs{\eM_{t \wedge \sigma_n} - V_{t\wedge \sigma_n}} \leq C(1+ L_{\sigma_n}(\epsilon)), \quad \text{ for } \epsilon \in (0,1].
\]
An application of Lemma~\ref{lem: loc time UI} ensures the uniform integrability of $\set{\sup_{t\in [0,T]}\abs{\eM_{t \wedge \sigma_n} - V_{t\wedge \sigma_n}}}_{\epsilon \in (0,1]}$. As a result, we obtain
\[
\lim_{\epsilon \downarrow 0} \expec \bra{\sup_{t \in [0, T]} \abs{\eM_{t \wedge
\sigma_n} - V_{t \wedge \sigma_n}}} = 0.
\]
Combining it with the martingale property of $\eM$ for each $\epsilon \in (0,1)$, we conclude that $V_{\cdot \wedge \sigma_n}$ is a martingale.
\end{proof}

Now we are ready to present the relationship between classical solutions and stochastic solutions.
\begin{prop}\label{prop: classical->stochastic}
 Any classical solution to \eqref{eq: pricing eq} is a stochastic solution.
\end{prop}

The following result will be useful in proving the above proposition.
\begin{lem}\label{lemma: local mart}
Let $\sigma$ be a stopping time and $Z$ be a nonnegative continuous-path process with $Z = Z_{\sigma \wedge \cdot}$. If there exists a nondecreasing sequence of stopping times $\set{\sigma_n}_{n\in \Natural}$ with $\prob \bra{\lim_{n\rightarrow \infty} \sigma_n = \sigma} = 1$ such that $Z_{\sigma_n \wedge \cdot}$ is a martingale for all $n \in \Natural$, then $Z$ is a local martingale.
\end{lem}
\begin{proof}
As $Z_{\sigma_n \wedge \cdot}$ is a nonnegative martingale, we have $\prob \bra{\sup_{t \in [0, \sigma_n]} Z_t > \ell} \leq 1 / \ell$ for all $n \in \Natural$ and $\ell \in \Real_+$. Since $Z = Z_{\sigma \wedge \cdot}$ and $\prob \bra{\lim_{n\rightarrow \infty} \sigma_n = \sigma} = 1$, we get $\prob \bra{\sup_{t \in \Real_+} Z_t < \infty} = 1$. Therefore, defining $\tsigma_k:= \inf \set{t \in \Real_+ \such Z_t \geq k}$ for $k \in \Natural$, we have $\prob \bra{\lim_{k\rightarrow \infty} \tsigma_k = \infty} = 1$. Furthermore, $\set{Z_{\sigma_n \wedge \tsigma_k}}_{n\in \Natural}$ is a uniformly integrable family for each $k$; indeed, this follows because $\prob \bra{\sup_{t\in [0, \tsigma_k]} Z_{t} \leq k} = 1$. We infer that $Z_{\tsigma_k  \wedge \cdot}$ is a martingale for each $k \in \Natural$, which concludes the proof.
\end{proof}

\begin{proof}[Proof of Proposition~\ref{prop: classical->stochastic}]
 For fixed $(x,y,T)\in \Real_{++}^3$, recall that $V^{x,y,T}_t = v(S^{x,y}_t, Y^y_t, T-t)$ for $t\in [0,T]$. We define $V^{x,y,T}_{\cdot}$ on $\Real_+$ via $V^{x,y,T}_{\cdot} = V^{x,y,T}_{\cdot \wedge T}$. Thanks to the previous lemma, to show that $V^{x,y,T}_{\cdot}$ is a local martingale on $[0,T]$, it suffices to find a sequence of stopping times $\set{\sigma_n}_{n\in \Natural}$ such that $\prob[\limn \sigma_n = T]=1$ and $V^{x,y,T}_{\cdot \wedge \sigma_n}$ is a martingale for each $n$. We shall use this observation to prove the statement in each case of Definition~\ref{def: classical soln}.

 \underline{Case (A)}: Consider $\sigma_n := \inf\set{t \in \Real_+ \such (S^{x,y}_t, Y^y_t) \notin [n^{-1}, n]^2}\wedge (T- T/n)$ for each $n \in \Natural$. Given a classical solution $v$, it follows from It\^{o}'s formula that $V^{x, y, T}_{\cdot \wedge \sigma_n}$ is a martingale. As $\prob \bra{\tau_0=\infty}=1$, $\prob \bra{\lim_{n\rightarrow \infty} \sigma_n = T} = 1$; therefore, $V_{\cdot}^{x, y, T}$ is a local martingale on $[0,T]$ thanks to Lemma~\ref{lemma: local mart}.

 \underline{Case (B)}: Given such a classical solution $v$, the same argument as in case (A) implies that $V^{x, y, T}_{\cdot \wedge \sigma_n}$ is a martingale on $[0,T]$ for any $n\in \Natural$. Since $\prob \bra{\tau^y_0 <\infty}>0$ in this case, we have $\prob \bra{\limn \sigma_n = \tau^y_0}=1$. However, the boundary condition \eqref{eq: boundary cond 2} implies that $V^{x, y, T}_{\cdot} = V^{x, y, T}_{\cdot \wedge \tau^y_0}$. Invoking Lemma \ref{lemma: local mart}, we conclude that $V^{x, y, T}$ is a local martingale.

 \underline{Case (C)}: Since $v\in \overline{\fC}$, there exists a sequence $\set{v^m}_{m\in \Natural}$ such that each $v^m \in \fC$ and $\set{v^m}_{m\in \Natural}$ converges to $v$ point-wise. It then follows from Lemma~\ref{lemma: class C} that $V^{m; x,y,T}_{\cdot \wedge \sigma_n} := v^m\pare{S^{x,y}_{\cdot \wedge \sigma_n}, Y^y_{\cdot \wedge \sigma_n}, T- \cdot \wedge \sigma_n}$ is a martingale on $[0,T]$, where $\sigma_n$ is defined in Lemma~\ref{lemma: class C}. On the other hand, since each $v^m$ is dominated by $h$, $V^{m; x,y,T}_{t\wedge \sigma_n} =v^m\pare{S^{x,y}_{t\wedge \sigma_n}, Y^y_{t\wedge \sigma_n}, T- t\wedge \sigma_n} \leq h\pare{S^{x,y}_{t\wedge \sigma_n}}$ for any $t\in [0,T]$ and $m\in \Natural$. Combining the previous inequality with $\expec\bra{h\pare{S^{x,y}_{t\wedge \sigma_n}}} \leq h\pare{\expec[S^{x,y}_T]} \leq h(x)<\infty$, we obtain that $\set{V^{m; x,y,T}_{t\wedge \sigma_n}}_{m\in \Natural}$ is a uniformly integrable family. Therefore, for any $s\in [0,t]$,
 \[
  \expec\bra{V^{x,y,T}_{t\wedge \sigma_n} \,|\, \F_s} = \lim_{m\to \infty} \expec\bra{V^{m; x,y,T}_{t\wedge \sigma_n} \,|\, \F_s} = \lim_{m\to \infty} V^{m; x,y,T}_{s \wedge \sigma_n} = V^{x,y,T}_{s\wedge \sigma_n},
 \]
 which confirms that $V^{x,y,T}_{\cdot \wedge \sigma_n}$ is a martingale on $[0,T]$. It is clear that $\prob[\lim_{n\to \infty} \sigma^n = T]=1$, then $V^{x,y,T}$ is a local martingale thanks to Lemma~\ref{lemma: local mart}.
\end{proof}

\subsection{Proofs of Theorems~\ref{thm: existence} and \ref{thm: uniqueness}}\label{sec:fsb}
\begin{proof}[Proof of Theorem~\ref{thm: existence}]
Proposition~\ref{prop: u classical soln} has already established that $u$ is a classical solution. The minimality property follows from  Propositions~\ref{prop: u classical soln}, \ref{prop: classical->stochastic}, and \ref{prop: existence}.
\end{proof}

\begin{proof}[Proof of Theorem~\ref{thm: uniqueness}]
 We will only prove the statement when $g$ is of linear growth, i.e., $\eta = \limsup_{x\to \infty} g(x)/x >0$, the proof for the strictly sublinear growth $g$ can be performed similarly.

 Given a classical solution $v$ dominated by $h$, $v$ is also a stochastic solution thanks to Proposition~\ref{prop: classical->stochastic}. Proposition~\ref{prop: uniqueness} implies that when  $S$ is a martingale, $v\equiv u$ on $\Real_+^3$. If $S$ is a strict local martingale, Proposition~\ref{prop: u classical soln} and Corollary~\ref{cor: delta} combined imply that both $u$ and $u+\eta \delta$ are both classical solutions dominated by $h$ (see \eqref{eq: u+eta delta}). However, they are different solutions since $\delta>0$. Then the statement in item (ii) is confirmed.

Let us consider the last statement of the theorem. It is clear that the comparison result implies the uniqueness of classical solutions. Conversely, when $g$ is of linear growth, we shall show that the martingale property of $S$ implies the comparison result. To this end, an argument similar to Proposition~\ref{prop: classical->stochastic} gives that $v\pare{S^{x,y}_\cdot, Y^y_\cdot, T-t}$ is a local supermartingale and $w\pare{S^{x,y}_\cdot, Y^y_\cdot, T-t}$ is a local submartingale, for any $(x,y,T)\in \Real_{++}^3$. Since they are both dominated by the martingale $M(1+ S^{x,y}_{\cdot})$, in fact, $v\pare{S^{x,y}_\cdot, Y^y_\cdot, T-t}$ is a supermartingale and $w\pare{S^{x,y}_\cdot, Y^y_\cdot, T-t}$  is a submartingale. As a result,
 \[
  v(x,y,T) \geq \expec\bra{v\pare{S^{x,y}_T, Y^y_T, 0}} \geq \expec\bra{g\pare{S^{x,y}_T}} \geq \expec\bra{w \pare{S^{x,y}_T, Y^y_T, 0}} \geq w(x,y,T), \quad \text{ on } \Real_{++}^3.
 \]
Finally, the inequality $v\geq w$ can be extended to $\Real_+^3$ thanks to the continuity of $v$ and $w$.
\end{proof}

\bibliographystyle{siam}
\bibliography{biblio}
\end{document}